%% file: main.tex
\documentclass[a4paper, 12pt]{article}
\usepackage{amssymb}
\usepackage{bm}
\usepackage{textcomp}
\usepackage{mathrsfs}
\usepackage{amsmath}
\usepackage{amsfonts}
\usepackage[T1]{fontenc}
\usepackage[utf8]{inputenc}    
\usepackage[T1]{fontenc}       
\usepackage{amsthm}
\usepackage{graphicx}
\usepackage{amsmath}

\usepackage{amsthm}
\usepackage{array}
\usepackage{graphicx}
\usepackage{cases}
\usepackage{enumerate}
\usepackage{clrscode}
\usepackage{listings}
\usepackage[ruled,vlined,english,linesnumbered]{algorithm2e}
\usepackage{url}
\usepackage{enumitem}
\usepackage{pgffor}
\usepackage{float}

\usepackage{enumerate,enumitem,todonotes}
\makeatletter
\def\th@plain{%
  \itshape 
}
\makeatother

\makeatletter
\renewenvironment{proof}[1][\proofname]{\par
  \pushQED{\qed}%
  \normalfont \topsep6\p@\@plus6\p@\relax
  \trivlist
  \item[\hskip\labelsep
        \bfseries
    #1\@addpunct{.}]\ignorespaces
}{%
  \popQED\endtrivlist\@endpefalse
}
\makeatother

\newtheorem{theorem}{Theorem}[section]

\numberwithin{equation}{section}

\newtheorem{lemma}[theorem]{Lemma}

\newtheorem{lem}[theorem]{Lemma}

\newtheorem{conj}[theorem]{Conjecture}

\newtheorem{question}[theorem]{Question}

\numberwithin{equation}{section}

\usepackage[varg]{txfonts}
\usepackage{graphicx, epsfig, subfigure}
\usepackage{enumerate} 
\usepackage[square, numbers, sort&compress]{natbib} 
\ifx\pdfoutput\undefined
 \usepackage[dvipdfm,%
  pdfstartview=FitH, 
  bookmarks=true,%
  bookmarksnumbered=true, 
  bookmarksopen=true, 
  plainpages=false,%
  pdfpagelabels,%
  colorlinks=true, 
  linkcolor=blue, 
  citecolor=blue,%
  urlcolor=black,
  pdfborder=001]{hyperref}
  \AtBeginDvi{}  
\else
 \usepackage[pdftex,%
  pdfstartview=FitH, 
  bookmarks=true,%
  bookmarksnumbered=true, 
  bookmarksopen=true, 
  plainpages=false,%
  pdfpagelabels,%
  colorlinks=true, 
  linkcolor=blue, 
  citecolor=blue,%
  urlcolor=black,
  pdfborder=001]{hyperref}
\fi

\numberwithin{equation}{section}

\begin{document}

\title{\LARGE Fast algorithm for $S$-packing coloring of Halin graphs}
\author{Xin Zhang \quad \quad Dezhi Zou\\
{\small School of Mathematics and Statistics, Xidian University, Xi'an, 710071, China}\\
{\small xzhang@xidian.edu.cn \quad  dezhizou@stu.xidian.edu.cn}
}


\maketitle

\begin{abstract}\baselineskip 0.60cm
Motivated by frequency assignment problems in wireless broadcast networks,
Goddard, Hedetniemi, Hedetniemi, Harris, and Rall introduced the notion of \(S\)-packing coloring  in 2008. Given a non-decreasing sequence \(S = (s_1, s_2, \ldots, s_k)\) of positive integers, an \(S\)-packing coloring of a graph \(G\) is a partition of its vertex set into \(k\) subsets \(\{V_1, V_2, \ldots, V_k\}\) such that for each \(1 \leq i \leq k\), the distance between any two distinct vertices \(u, v \in V_i\) is at least \(s_i + 1\). In this paper, we study the \(S\)-packing coloring problem for Halin graphs with maximum degree \(\Delta \leq 5\). Specifically, we present a linear-time algorithm that constructs a \((1,1,2,2,2)\)-packing coloring for any Halin graph satisfying \(\Delta \leq 5\). It is worth noting that there are Halin graphs that are not $(1,2,2,2)$-packing colorable.

\vspace{3mm}\noindent \emph{Keywords}: vertex coloring; packing coloring; Halin graph.
\end{abstract}

\baselineskip 0.60cm

\section{Introduction}\label{sec1}

Let $ S = (s_1, s_2, \ldots, s_k) $ be a non-decreasing sequence of positive integers (i.e., $ s_1 \leq s_2 \leq \cdots \leq s_k $). An \textit{$ S $-packing coloring} of a simple graph $ G = (V(G), E(G)) $ is a partition of $ V(G) $ into $ k $ disjoint subsets $ V_1, V_2, \ldots, V_k $ such that for each $ i \in [k]:=\{1, 2, \ldots, k\} $, any two distinct vertices $ u, v \in V_i $ satisfy $ d_G(u, v) \geq s_i + 1 $, where $ d_G(u, v) $ denotes the distance between $ u $ and $ v $ in $ G $. The \textit{packing chromatic number} of $ G $, denoted $ \chi_p(G) $, is the smallest integer $ k $ for which $ G $ admits a $(1, 2, \ldots, k)$-packing coloring (also called a \textit{$ k $-packing coloring}). The concept of packing chromatic number was introduced by Goddard, Hedetniemi, Hedetniemi, Harris, and Rall \cite{goddard2008broadcast} in 2008 under the original name \emph{broadcast chromatic number}, motivated by frequency assignment problems in wireless broadcast networks.
 
A \emph{proper $k$-coloring} of $G$ is a mapping $\varphi: V(G) \to [k]$ such that for any adjacent vertices $u, v \in V(G)$, $\varphi(u) \neq \varphi(v)$; if such a mapping exists, then $G$ is said to be \emph{$k$-colorable}. In contrast, a \emph{2-distance coloring} of $G$ is a vertex coloring $\varphi: V(G) \to [k]$ such that for any two distinct vertices $u, v \in V(G)$ with $d_G(u, v) \leq 2$, $\varphi(u) \neq \varphi(v)$; the minimum $k$ for which this is possible is the \emph{2-distance chromatic number} $\chi_2(G)$. Notably, these concepts relate to packing colorings: a $(1, 1, \ldots, 1)$-packing coloring corresponds to an ordinary proper coloring, while a $(2, 2, \ldots, 2)$-packing coloring aligns with 2-distance coloring. Furthermore, the $(2, 2, \ldots, 2)$-packing coloring is closely related to, but strictly stronger than, the concept of \textit{exact 2-distance coloring} \cite[Section 11.9]{nesetril2012sparsity}: while an exact 2-distance coloring requires only that vertices at distance exactly 2 receive distinct colors (allowing adjacent vertices to share colors), a $(2, 2, \ldots, 2)$-packing coloring enforces that all pairs of vertices within each color class are at distance at least 3, thereby implying a valid exact 2-distance coloring but not vice versa. If graphs are restricted to be triangle-free, then exact 2-distance coloring becomes equivalent to an \textit{injective coloring}---a coloring where any two vertices sharing a common neighbor receive distinct colors \cite{zbMATH07662009,zbMATH05572547}. It follows that even in triangle-free graphs, $(2, 2, \ldots, 2)$-packing coloring remains strictly stronger than injective coloring. It is worth noting that Bu, Qi and Zhu \cite{zbMATH07367784} proved that Halin graphs with maximum degree at most five has an injective 7-coloring.
We refer the readers to further reading on these relative topics \cite{VANDENHEUVEL2019143,QUIROZ2020111769,PRIYAMVADA202284,FOUCAUD202174}.

We define $S(G)$ as the graph obtained by \emph{subdividing each edge} of $G$, i.e., replacing each edge $uv \in E(G)$ with a path $u$-$w$-$v$ through a new vertex $w$. The problem of bounding $\chi_p(G)$ and $\chi_p(S(G))$ for subcubic graphs (where each vertex has degree at most 3) has been investigated in several studies \cite{r4,r6}. 
Gastineau and Togni \cite{r6} were the first to ask whether $\chi_p(S(G)) \leq 5$ holds for all subcubic graphs $G$. This question was later formalized as a conjecture by Bre\v{s}ar et al. \cite{r4}, who proposed the following:
 
\begin{conj}\label{conj1}
   If $G$ is a subcubic graph, then $\chi_p(S(G)) \leq 5$.
\end{conj}

For a graph \( G \) and its subdivision \( S(G) \), Gastineau and Togni \cite{r6} established a fundamental relationship between their packing colorability: if \( G \) is \( (s_1, s_2, \ldots, s_k) \)-packing colorable, then \( S(G) \) is \( (1, 2s_1 + 1, 2s_2 + 1, \ldots, 2s_k + 1) \)-packing colorable. They further observed that when \( G \) is \( (1^2,2^2) \)-packing colorable  ($p^r$ in a sequence means $p$ is repeated $r$
times), this immediately implies \( \chi_p(S(G)) \leq 5 \). Consequently, to verify Conjecture~\ref{conj1} for any subclass of subcubic graphs, it suffices to demonstrate \( (1^2,2^2) \)-packing colorability within that subclass. 

Toward this goal, Liu, Zhang, and Zhang \cite{r2} proved that every subcubic graph \( G \) is \( (1^2,2^2,3) \)-packing colorable, thus establishing \( \chi_p(S(G)) \leq 6 \). This result is just one parameter shy of the conjectured bound and significantly improved an earlier result by Balogh, Kostochka, and Liu \cite{BKL2}, who showed that the packing chromatic number of the 1-subdivision of subcubic graphs is bounded by 8.

Several studies have investigated diverse packing \( S \)-colorings of subcubic graphs \cite{zbMATH07240148, zbMATH07352780, zbMATH07677301, zbMATH07215509, zbMATH07939373, zbMATH07983135, zbMATH07051059, zbMATH06007892, zbMATH06689498, zbMATH07377172}. In particular, Gastineau and Togni \cite{r6} demonstrated that every subcubic graph admits both \((1^2, 2^3)\)-packing colorings and \((1, 2^6)\)-packing colorings. Balogh, Kostochka, and Liu \cite{BKL2} further showed that any subcubic graph has a $(1^2,2^2,3^2,k)$-packing coloring where color $k \geq 4$ is used at most once, and that every 2-degenerate subcubic graph is $(1^2,2^2,3^2)$-packing colorable. 

For cubic graphs specifically, Cranston and Kim \cite{CK1} proved that all cubic graphs except the Petersen graph admit $(2^8)$-packing colorings. This result was complemented by Thomassen \cite{T1} and independently by Hartke, Jahanbekam, and Thomas \cite{HJT1}, who established that cubic planar graphs are $(2^7)$-packing colorable.

To systematically refine these results, researchers have focused on \emph{$i$-saturated subcubic graphs}, a family defined by structural constraints on high-degree vertices. A \textit{$k$-degree graph} is a graph $G$ satisfying $ \Delta(G) \leq k$, and an \textit{$i$-saturated subcubic graph} is one where each vertex of degree 3 is adjacent to at most $i$ other degree-3 vertices. This classification has yielded significant progress: Yang and Wu \cite{zbMATH07684692} proved that all 0-saturated subcubic graphs admit $ (1^2,3) $-packing colorings; Bazzal \cite{arXiv:2409.01769} extended this to 1-saturated graphs with $(1^2,2) $-colorings; Mortada and Togni \cite{zbMATH07950632} further generalized to 2-saturated graphs with $ (1^2,2,3) $-colorings; and their subsequent work \citep{zbMATH07825852} established the $(1^2, 2^2)$-colorability for $(3,0)$-saturated variants (a refined subclass where $(3,0)$-saturated denotes 3-saturated graphs with non-adjacent heavy vertices, defined as degree-3 vertices whose neighbors are all degree-3 vertices).


Despite the limited exploration of $S$-packing coloring in graphs with larger maximum degrees, Mortada and Togni \cite{arXiv:2503.18793} recently advanced this field by analyzing $k$-degree graphs under varying saturation constraints, where a $k$-degree graph $G$ is \emph{$t$-saturated} if every vertex of degree $k$ has at most $t$ neighbors of degree $k$. Their results are as follows: for $0$-saturated $k$-degree graphs, $G$ admits a $(1^{k-1}, 3)$-packing coloring; when saturation increases to $1 \leq t \leq k-2$, it transitions to $(1^{k-1}, 2)$-packing colorability. Further, $(k-1)$-saturated $k$-degree graphs ($k \geq 4$) require a $(1^{k-1}, 2^{k-1})$-packing coloring, while all $k$-degree graphs ($k \geq 3$) universally satisfy $(1^{k-1}, 2^{k})$-packing colorability. These findings not only generalize prior work on subcubic graphs but also offer a unified framework for analyzing degree-constrained graph colorability.

A \textit{Halin graph} is a plane graph constructed from a tree $T$ (called \textit{characteristic tree}) without vertices of degree two by connecting all leaves through a cycle $C$ (called \textit{adjoint cycle}). Halin graphs exhibit distinct $S$-packing colorability patterns. Tarhini and Togni \cite{r1} proved cubic Halin graphs are $(1^2,2,3)$-packing colorable, while Wang, Hu, and Wang \cite{r3} showed 5-degree Halin graphs require $(2^6)$-packing colorings due to their square graphs' chromatic number 6. We strengthen this bound in this paper, proving that 5-degree Halin graphs actually admit $(1^2,2^3)$-packing colorings, improving the state-of-the-art for this graph class. It is worth noting that the wheel graph $W_5$ (with maximum degree $5$) is a Halin graph that  does not admit a $(1,2^3)$-packing coloring.

\begin{theorem}\label{thm1}
Halin graphs with maximum degree at most 5 are $(1^2,2^3)$-packing colorable.
\end{theorem}

The proof of Theorem~\ref{thm1} is algorithmic: it provides an explicit procedure to construct the desired $(1^2,2^3)$-packing coloring in $O(n)$ time, where $n$ denotes the number of vertices in the Halin graph. The algorithm is introduced in Section~\ref{sec:2} and verified in Section~\ref{sec:3}. In Section~\ref{sec:4}, we compile a list of open problems concerning $S$-packing colorings of Halin graphs.

\section{A quick algorithm} \label{sec:2}

Let $G$ be a Halin graph with maximum degree at most 5, where 
$T$ is the characteristic tree and $C:=a_1a_2\cdots a_na_1$ is the adjoint cycle of $G$.
The set of colors that we are going to use is $\{1,1',2_a,2_b,2_c\}$.  

First, we properly color the tree $T$ using two colors $1$ and $1'$, and denote this coloring by $\varphi$. Next, we remove all colors from $C$ and recolor the vertices $a_1, a_2, \ldots, a_n$ according to the \hyperref[alg-recolor]{\texttt{RECOLORING}} algorithm.
In this algorithm, we introduce the boolean variable $\texttt{all\_same}$ (initialized to $\mathsf{True}$) to distinguish between two cases. If either $\varphi(a_1) \neq \varphi(a_n)$ initially holds or there exists an index $2\leq i \leq n$ such that $\varphi(a_i) \neq \varphi(a_{i-1})$, we perform a \emph{cyclic relabeling} of the vertices $a_1, a_2, \ldots, a_n$ to ensure $\varphi(a_1) \neq \varphi(a_n)$, reset $\texttt{all\_same} $ to be $\mathsf{False}$, and apply \hyperref[alg-case-1]{\texttt{CASE-1}}.
Otherwise (i.e., when $\varphi(a_i) = 1$ for all $i \in [n]$), the value of $\texttt{all\_same}$ remains True and we directly apply \hyperref[alg-case-2]{\texttt{CASE-2}}.
Note that each $a_i$ is a leaf in $T$, and let $b_i$ be its unique neighbor in $T$.
For any $i$, denote by $1_i$ the color in $\{1, 1'\} \setminus \varphi(b_i)$, and let $\mathcal{N}[b_i]$ denote the set of vertices $a_j$ adjacent to $b_i$.

\begin{lem}
\label{claim-1}
    For each $b_i$ with $i\in [n]$, $\mathcal{N}[b_i]$ contains at most three vertices.
\end{lem}

\begin{proof}
Suppose, for the sake of contradiction, that $\{a_i, a_j, a_k, a_\ell\} \subseteq \mathcal{N}[b_i]$, where $i < j < k < \ell$. Given that $\{\varphi(a_i), \varphi(a_j), \varphi(a_k), \varphi(a_\ell)\} \subseteq \{2_a, 2_b, 2_c\}$, by the \hyperref[alg-recolor]{\texttt{RECOLORING}} algorithm, at most one of the differences $j-i$, $k-j$, and $\ell-k$ can be equal to 1. Without loss of generality, assume that $j-i \neq 1$ and $k-j \neq 1$. This implies the existence of vertices $a_{i'}$ and $a_{j'}$ such that $i < i' < j < j' < k$.
Since $a_{i'}$ and $a_{j'}$ are leaves of the tree $T$, there must exist paths from $b_i$ to $a_{i'}$ and from $b_i$ to $a_{j'}$ within $T$. Consequently, $b_i$ must be incident with at least six edges, which contradicts the condition $\Delta(G) \leq 5$.    
\end{proof}

\begin{algorithm}[H]\label{alg-recolor}
    \caption{\textbf{RECOLORING($\varphi(a_1),\varphi(a_2),\cdots,\varphi(a_n)$)}}
    
    \LinesNumbered
    \DontPrintSemicolon
    
    \KwIn{Colors $\varphi(a_1),\varphi(a_2),\cdots,\varphi(a_n)$ of vertices $a_1,a_2,\cdots,a_n$}
    \KwOut{Colors of $a_1,a_2,\cdots,a_n$ after recoloring}

    $all\_same \gets \mathsf{True}$ \\
    
 \If{$\varphi(a_1) = \varphi(a_n)$}
{
    \For{$i \gets 2$ \textbf{to} $n$ \label{algo:recolor-line3}}
    {
        \If{$\varphi(a_i) \neq \varphi(a_{i-1})$}
        {
            \For{$j \gets 0$ \textbf{to} $n-1$}
            {
                $a_{j+1} \gets a_{(i+j-1 \mod n)+1}$ \label{algo:recolor-line6}
            }
            $all\_same \gets \mathsf{False}$ \\
            \textbf{break} 
        }
    }
}
\Else
{
    $all\_same \gets \mathsf{False}$
}

\If{$all\_same=\mathsf{False}$}{
\textbf{CASE-1}($a_1,a_2,\cdots,a_n,\varphi(a_1),\varphi(a_2),\cdots,\varphi(a_n)$)\\
        }
    \Else
    {
        \textbf{CASE-2}($a_1,a_2,\cdots,a_n,\varphi(a_1),\varphi(a_2),\cdots,\varphi(a_n)$)

    }
    
\end{algorithm}

\begin{algorithm}[H]\label{alg-case-1}
    \caption{\textbf{CASE-1($a_1,a_2,\cdots,a_n,\varphi(a_1),\varphi(a_2),\cdots,\varphi(a_n)$)}}
    
    \LinesNumbered
    \DontPrintSemicolon
    
  \KwIn{Colors $\varphi(a_1),\varphi(a_2),\cdots,\varphi(a_n)$ of vertices $a_1,a_2,\cdots,a_n$}
    \KwOut{Colors of $a_1,a_2,\cdots,a_n$ after recoloring}
    
$m \gets \lfloor n / 4 \rfloor * 4$ \tcp*{Ensure m is the largest multiple of 4 less than or equal to n}

\For{$k \gets 1$ \textbf{to} $m$}{
            \uIf{$k \equiv 1 \pmod{2}$}{
                $\varphi(a_k) \gets 1_k$ \tcp*{Color $a_k$ from $\{1,1'\}$ avoiding the color on $b_k$}
            }
            \uElseIf{$k \equiv 2 \pmod{4}$}{
                $\varphi(a_k) \gets 2_a$
            }
            \Else{
                $\varphi(a_k) \gets 2_b$ \tcp*{Covers $n \equiv 0 \pmod{4}$ case}
            }
        }

\Switch{$(n - m)$}{
\Case{1}{
    $\varphi(a_n) \gets 1_n$ \tcp*{Covers $n \equiv 1 \pmod{4}$ case}
}
\Case{2}{
    $\varphi(a_{n-1}) \gets 2_a$; $\varphi(a_n) \gets 1_n$ \tcp*{Covers $n \equiv 2 \pmod{4}$ case}
}
\Case{3}{
    $\varphi(a_{n-2}) \gets 1_{n-2}$; $\varphi(a_{n-1}) \gets 2_a$; $\varphi(a_n) \gets 1_n$ \tcp*{Covers $n \equiv 3 \pmod{4}$ case}
}
}
\end{algorithm}

\begin{algorithm}[H]\label{alg-case-2}
    \caption{\textbf{CASE-2($a_1,a_2,\cdots,a_n,\varphi(a_1),\varphi(a_2),\cdots,\varphi(a_n)$)}}
    
    \LinesNumbered
    \DontPrintSemicolon
    
  \KwIn{Colors $\varphi(a_1),\varphi(a_2),\cdots,\varphi(a_n)$ of vertices $a_1,a_2,\cdots,a_n$}
    \KwOut{Colors of $a_1,a_2,\cdots,a_n$ after recoloring}
    
$m \gets \lfloor n / 4 \rfloor * 4$ \tcp*{Ensure m is the largest multiple of 4 less than or equal to n}

\For{$k \gets 1$ \textbf{to} $m$}{
            \uIf{$k \equiv 1 \pmod{2}$}{
                $\varphi(a_k) \gets 1_k$ \tcp*{Color $a_k$ from $\{1,1'\}$ avoiding the color on $b_k$}
            }
            \uElseIf{$k \equiv 2 \pmod{4}$}{
                $\varphi(a_k) \gets 2_a$
            }
            \Else{
                $\varphi(a_k) \gets 2_b$ \tcp*{Covers $n \equiv 0 \pmod{4}$ case}
            }
        }

\Switch{$(n - m)$}{
\Case{1}{
    $\varphi(a_n) \gets 2_c$ \tcp*{Covers $n \equiv 1 \pmod{4}$ case}
}
\Case{2}{
    $\varphi(a_{n-1}) \gets 1_{n-1}$; $\varphi(a_n) \gets 2_c$ \tcp*{Covers $n \equiv 2 \pmod{4}$ case}
}
\Case{3}{
    $\varphi(a_{n-2}) \gets 2_{a}$; $\varphi(a_{n-1}) \gets 1_{n-1}$; $\varphi(a_n) \gets 2_b$ \tcp*{Covers $n \equiv 3 \pmod{4}$ case}
}
}

        \tcp*{Check for specific condition and modify colors accordingly}
        \For{$i \gets 1$ \textbf{to} $n-2$  \label{algo:recolor-line12}}{
        \If{$\varphi(a_n) = 2_c ~ \textbf{and}~ \varphi(a_{n-1}) = \varphi(a_i) = 2_b ~ \textbf{and}~  b_n=b_{n-1}=b_i$ \label{algo:recolor-line13}}{
            $\varphi(a_2) \gets 2_c$; $\varphi(a_{n-1}) \gets 2_c$; $\varphi(a_n) \gets 2_a$\\
            \textbf{break}  \label{algo:recolor-line15}
        }
        }
        $not\_separated \gets \mathsf{True}$ \label{algo:recolor-line16}\\
        \If{$\varphi(a_n)=2_c$}{
            \For{$i \gets 1$ \textbf{to} $n-1$  }{
            \If{$\varphi(a_i) \in \{ 2_a,2_b\}
             ~ \textbf{and}~   b_n=b_i$ }{
                $not\_separated\gets \mathsf{False}$\\
                \textbf{break}
                }
            }
            \If{$not\_separated$}{
            \For{$i \gets 4$ \textbf{to} $n-1$  }{
            \If{$\varphi(a_i) = 2_a  ~ \textbf{and}~    b_i=b_2$ }{
                $\varphi (a_2) \gets 2_c; \varphi (a_n) \gets 2_a$
                
                \textbf{break} \label{algo:recolor-line26}
                }
            }
            }
        }

\end{algorithm}

For a sequence of $n$ vertices $a_1, a_2, \ldots, a_n$, its \textit{color sequence} is a string $\ell_1 \ell_2 \cdots \ell_n \in \Sigma^n$, where $\Sigma \subseteq \mathbb{N}$ is a finite set of colors, and $\ell_j$ denotes the color of vertex $a_j$.
In this paper, if the index $i$ appears in a color sequence as $\ell_j$ (i.e., $\ell_j = i$), its interpretation is determined by the color of vertex $b_j$: $i$ represents the color $1$ if $b_j$ has color $1'$, and otherwise represents the color $1'$.

The following lemma is a direct consequence of Algorithms \hyperref[alg-case-1]{\texttt{CASE-1}} and \hyperref[alg-case-2]{\texttt{CASE-2}}, and is presented here without proof. 
It is worth noting that \hyperref[alg-case-2]{\texttt{CASE-2}} first generates a color sequence for vertices $a_1, a_2, \ldots, a_n$, following which certain vertices are recolored under distinct conditions. 

\begin{itemize}
    \item If there exists a unique vertex $b_i \in V(T)$ (guaranteed by Lemma \ref{claim-1}) with $\mathcal{N}[b_i] = \{a_i, a_{n-1}, a_n\}$ and the color sequence of $a_i,a_{n-1},a_n$ is $2_b2_b2_c$, we perform the following recoloring in lines \ref{algo:recolor-line12}--\ref{algo:recolor-line15} of \hyperref[alg-case-2]{\texttt{CASE-2}}: reassign $2_c$ to both $a_i$ and $a_{n-1}$, and $2_a$ to $a_n$ (see Figure \ref{fig:obs}(c) for illustration).
    \item If $\mathcal{N}[b_n]$ contains no vertices colored with $2_a$ or $2_b$, and $b_2$ is adjacent to some $a_j \neq a_2$ colored $2_a$, then we recolor $a_2$ to $2_c$ and $a_n$ to $2_a$ (see Figures \ref{fig:obs}(d) or \ref{fig:obs}(g) for examples). 
    In the algorithm, we introduce a Boolean variable $\texttt{not\_separated}$ (initialized to $\mathsf{True}$) to determine whether $\mathcal{N}[b_n]$ contains no vertices colored $2_a$ or $2_b$: if $\texttt{not\_separated}$ becomes $\mathsf{False}$, the algorithm terminates without action; otherwise, it checks if $b_2$ is adjacent to some $a_j \neq a_2$ colored $2_a$, and if this condition holds, performs the recoloring operation (see lines \ref{algo:recolor-line16}--\ref{algo:recolor-line26} of \hyperref[alg-case-2]{\texttt{CASE-2}} for implementation details).
\end{itemize}

\begin{lem}
    \label{coloring}
If the \hyperref[alg-case-1]{\texttt{CASE-1}} algorithm is applied, the colors assigned to $a_1, a_2, \ldots, a_n$ are as follows:
\begin{description}[font=\normalfont\itshape, style=unboxed, leftmargin=*, itemsep=0.5ex]
    \item[when $n \equiv 0 \pmod{4}$:] the color sequence is $i2_a i2_b \cdots i2_a i2_b$.
    \item[when $n \equiv 1 \pmod{4}$:] the color sequence is $i2_a i2_b \cdots i2_a i2_b i$.
    \item[when $n \equiv 2 \pmod{4}$:] the color sequence is $i2_a i2_b \cdots i2_a i2_b 2_a i$.
    \item[when $n \equiv 3 \pmod{4}$:] the color sequence is $i2_a i2_b \cdots i2_a i2_b i2_a i$.
\end{description}

If the \hyperref[alg-case-2]{\texttt{CASE-2}} algorithm is applied, the colors are determined by more complex conditions:
\begin{description}[font=\normalfont\itshape, style=unboxed, leftmargin=*, itemsep=0.5ex]
    \item[when $n \equiv 0 \pmod{4}$:] the color sequence is $i2_a i2_b \cdots i2_a i2_b$.
    
    \item[when $n \equiv 1 \pmod{4}$:] the color sequence is one of the following, depending on the graph structure:
    \begin{itemize}[leftmargin=1.5em, itemsep=0.3ex]
        \item $i2_a i2_b \cdots i2_a i2_b 2_c$, \\
        which occurs if and only if 
        \begin{itemize}
            \item either $\mathcal{N}[b_n]$ has a vertex $a_j$ colored with $2_a$ or $2_b$, and when $j=n-1$, no other vertex in $\mathcal{N}[b_n]$ is colored with $2_b$ (see Figure~\ref{fig:obs}(a)), 
            \item or $b_2$ is not adjacent to an $a_j \neq a_2$ colored with $2_a$ (see Figure~\ref{fig:obs}(b));
        \end{itemize}

        \item $i2_c i2_b i2_a i2_b \cdots i2_a i2_b i2_a i2_c 2_a$, \\
        which occurs if and only if $b_n$ is adjacent to both $a_{n-1}$ and $a_n$, and also to another vertex $a_j$ colored with $2_b$ (see Figure~\ref{fig:obs}(c));

        \item $i2_c i2_b i2_a i2_b \cdots i2_a i2_b 2_a$, \\
     which occurs if and only if $\mathcal{N}[b_n]$ has no vertex colored with $2_a$ or $2_b$, and $b_2$ is adjacent to an $a_j \neq a_2$ colored with $2_a$ (see Figure~\ref{fig:obs}(d)); 
    \end{itemize}
    
    \item[when $n \equiv 2 \pmod{4}$:] the color sequence is one of the following:
    \begin{itemize}[leftmargin=1.5em, itemsep=0.3ex]
        \item $i2_a i2_b \cdots i2_a i2_b i2_c$, \\
   with conditions identical to the first case of $n \equiv 1 \pmod{4}$ (see Figure~\ref{fig:obs}(e) and (f));
    
        \item $i2_c i2_b i2_a i2_b \cdots i2_a i2_b i2_a$, \\
with conditions identical to the third case of $n \equiv 1 \pmod{4}$ (see Figure~\ref{fig:obs}(g)).
    \end{itemize}
    
    \item[when $n \equiv 3 \pmod{4}$:] the color sequence is $i2_a i2_b \cdots i2_a i2_b 2_a i2_b$. \hfill$\square$
\end{description}
\end{lem} 

     \begin{figure}
        \centering
        \scalebox{1.0}{
            \input{observation}
        }
        \caption{The (re)coloring procedure involved in Algorithm \hyperref[alg-case-2]{\texttt{CASE-2}} }
        \label{fig:obs}
    \end{figure}

    \begin{lem} \label{claim-2}
    For each $b_i$ with $i\in [n]$, if $\mathcal{N}[b_i]$ contains exactly three vertices, then at least two of them are neighbor to each other on $C$, and they are not in the same color.
\end{lem}

\begin{proof}
Let $\mathcal{N}[b_i] = \{a_i, a_j, a_k\}$. Without loss of generality, assume $i < j < k$. Now, either $j - i = 1$ or $k - j = 1$ must hold. Otherwise, by an argument analogous to that in Lemma \ref{claim-1}, we would deduce that $b_i$ must have a degree of at least six, which contradicts the given condition $\Delta(G) \leq 5$.
Given that $\{\varphi(a_i), \varphi(a_j), \varphi(a_k)\} \subseteq \{2_a, 2_b, 2_c\}$, the \hyperref[alg-recolor]{\texttt{RECOLORING}} algorithm implies that the colors assigned to $a_i, a_j, a_k$ are one of the following by Lemma \ref{coloring}:
\begin{description}[font=\normalfont, style=unboxed, leftmargin=*, itemsep=0ex]
    \item[when $n \equiv 1 \pmod{4}$, $j = n-1$, and $k = n$:] the color sequence is either $2_a 2_b 2_c$ or $2_b 2_c 2_a$;
    \item[when $n \equiv 2 \pmod{4}$, $j = n-1$, and $k = n$:] the color sequence is either $2_a 2_b 2_a$ or $2_b 2_b 2_a$;
    \item[when $n \equiv 3 \pmod{4}$, $i = n-3$, $j = n-2$, and $k = n$:] the color sequence is $2_b 2_a 2_b$;
    \item[when $n \equiv 3 \pmod{4}$, $j = n-3$, and $k = n-2$:] the color sequence is either $2_a 2_a 2_b$ or $2_b 2_a 2_b$.
\end{description}
In either case, elements in $\mathcal{N}[b_i] $ are not in the same color.
\end{proof}

Next, we apply the \hyperref[alg1]{\texttt{CONFLICTS-RESOLVING}} algorithm to resolve potential pairs of conflicts arising from the colors $2_a$ or $2_b$ 
by recoloring additional vertices with the color $2_c$. In the beginning of the \hyperref[alg1]{\texttt{CONFLICTS-RESOLVING}} algorithm, we determine $\mathcal{N}[b_i]$ for every $1 \leq i \leq n$ by building a dictionary $\mathcal{N}$. Given a dictionary $\mathcal{N}: K \to Y$ (where $K$ is the set of keys and $Y$ is the set of values), the \textit{domain} is defined as:
\[
\text{dom}(\mathcal{N}) = \{ k \in K \mid \exists y \in Y, \mathcal{N}(k) = y \}.
\]
This means $\text{dom}(\mathcal{N})$ is the set of all keys $k$ for which $\mathcal{N}[k]$ has an associated value. One can easily see that the keys for the dictionary $\mathcal{N}$ here are $b_1, b_2, \ldots, b_n$. 
After this, we check for every even $i$ whether $\mathcal{N}[b_i]$ contains a pair of vertices $a_i$ and $a_j$ such that they are both colored $2_a$ or $2_b$. Note that if such a pair exists, then $j$ is uniquely determined by $i$ according to Lemma \ref{claim-2}; thereafter, we recolor $a_i$ with $2_c$ if $\varphi(a_{i-2}) \neq 2_c$, or $a_j$ with $2_c$ otherwise.

\begin{algorithm}[H]\label{alg1}
    \caption{\textbf{CONFLICTS-RESOLVING($\varphi(a_1),\varphi(a_2),\cdots,\varphi(a_n)$)}}
    
\LinesNumbered
\DontPrintSemicolon

\KwIn{Colors $\varphi(a_1),\varphi(a_2),\cdots,\varphi(a_n)$ of vertices $a_1,a_2,\cdots,a_n$}
\KwOut{Colors of $a_1,a_2,\cdots,a_n$ after recoloring}

$\mathcal{L} \gets \{(a_1, b_1), (a_2, b_2), \dots, (a_n, b_n)\}$ \\

     $\mathcal{N} \gets \emptyset$ \tcp*{Initialize an empty dictionary}
    
    \ForEach{$(a, b) \in \mathcal{L}$ \textbf{and} $\varphi(a) \notin \{1, 1'\}$}{
        \eIf{$b \in \text{dom}(\mathcal{N})$}{ \tcp*{Check if $b$ is a key in $\mathcal{N}$}
            $\mathcal{N}[b] \gets \mathcal{N}[b] \cup \{a\}$ \tcp*{Add $a$ to the set of $b$'s neighbors}
        }{
            $\mathcal{N}[b] \gets \{a\}$ \tcp*{Initialize $b$'s neighbor set with $\{a\}$}
        }
    }
    
    \Return{$\mathcal{N}$} \tcp*{Return the constructed dictionary}

$a_0 \gets a_n$ \\
\For{$i \gets 1$ \textbf{to} $n$ \label{algo4:line9}}{
    \If{$i \equiv 0 \pmod{2}$}{
        $J \gets \{j~|~a_j \in \mathcal{N}[b_i]\} \setminus \{i\}$\\ 
        \ForEach{$j$ in $J$}{
            \If{$\varphi(a_i) = \varphi(a_j) ~ \textbf{and}~  \varphi(a_i) \neq 2_c$}{
                \If{$\varphi(a_{i-2}) \neq 2_c$}{
                    $\varphi(a_i) \gets 2_c$
                }
                \If{$\varphi(a_{i-2}) = 2_c$ \label{algo4:line16}}{
                    $\varphi(a_j) \gets 2_c$ \label{algo4:line17}
                }
                \tcp*{Exit the inner loop after processing}
                \textbf{break}
            }
        }
    }
}
\end{algorithm}

Finally, we integrate the \hyperref[alg-recolor]{\texttt{RECOLORING}} algorithm and the \hyperref[alg1]{\texttt{CONFLICTS-RESOLVING}} algorithm to formulate the \hyperref[alg-final]{\texttt{PACKING-COLORING}} algorithm.
This ultimately yields the desired $S$-packing coloring of the Halin graph $G$, satisfying all distance constraints specified by the sequence $S:=(1^2,2^3)$.
However, this conclusion is not self-evident.
This will be verified in the next section.

 \begin{algorithm}[H]\label{alg-final}
        \caption{\textbf{PACKING-COLORING($G$)}}
        \LinesNumbered
\DontPrintSemicolon
        \KwIn{A Halin graph $G = T\cup C$ with $\Delta(G)\leq 5$, where $T$ is the characteristic tree and $C:=a_1a_2\cdots a_na_1$ is the adjoint cycle of $G$.}

        \KwOut{A $(1^2,2^3)$-packing coloring of $G$}

\tcp*{Construct a 2-coloring $\varphi$ of $T$ by two colors $1$ and $1'$}
Initialize an empty coloring $\varphi$ for all vertices in $T$ \label{algo:pc-line1} \\
Choose an arbitrary root vertex $r$ in $T$ \\
Perform a BFS starting from $r$ to traverse the tree $T$ \\
\ForEach{vertex $v$ in $T$ in the order of traversal}{
    \If{$v$ is the root $r$}{
        $\varphi(v) \gets 1$
    }
    \Else{
        \If{the parent of $v$ has color $1$}{
            $\varphi(v) \gets 1'$
        }
        \Else{
            $\varphi(v) \gets 1$  \label{algo:pc-line11}
        }
    }
}
 
\textbf{RECOLORING($\varphi(a_1),\varphi(a_2),\cdots,\varphi(a_n)$)} \label{algo:pc-line12}\\
\textbf{CONFLICTS-RESOLVING($\varphi(a_1),\varphi(a_2),\cdots,\varphi(a_n)$)}  \label{algo:pc-line13}
\end{algorithm}

\section{Algorithm Verification and Complexity} \label{sec:3}

The \hyperref[alg-final]{\texttt{PACKING-COLORING}} algorithm outputs a coloring of $G$ using colors from $\{1,1',2_a,2_b,2_c\}$. To verify this is a $(1^2,2^3)$-packing coloring, we show a series of lemmas as follows.

\begin{lem} \label{lem:11'}
Vertices colored with $1$ (or $1'$) induce an independent set.
\end{lem}

\begin{proof}
   The \hyperref[alg-final]{\texttt{PACKING-COLORING}} algorithm begins by properly coloring the characteristic tree $T$ using two colors, $1$ and $1'$. Upon reaching line \ref{algo:pc-line12}, the algorithm proceeds to recolor all vertices on the adjoint cycles according to two distinct cases. In each case, when a vertex $a_k$ from the set $\{1,1'\}$ is recolored using either the \hyperref[alg-case-1]{\texttt{CASE-1}} or \hyperref[alg-case-2]{\texttt{CASE-2}} algorithm, the color chosen is distinct from the one currently assigned to $b_k$. Throughout the execution of the algorithms \hyperref[alg-case-1]{\texttt{CASE-1}}, \hyperref[alg-case-2]{\texttt{CASE-2}}, and \hyperref[alg-recolor]{\texttt{RECOLORING}}, no two consecutive vertices in the adjoint cycle $C$ are assigned the same color. It is noteworthy that, in the \hyperref[alg-case-1]{\texttt{CASE-1}} algorithm, vertex $a_1$ is reassigned the color $1_1$, while vertex $a_n$ is always reassigned the color $1_n$. Given that the original colors of $a_1$ and $a_n$, as determined by lines \ref{algo:pc-line1}--\ref{algo:pc-line11} of the \hyperref[alg-final]{\texttt{PACKING-COLORING}} algorithm and lines \ref{algo:recolor-line3}--\ref{algo:recolor-line6} of the \hyperref[alg-recolor]{\texttt{RECOLORING}} algorithm, are distinct, it follows that $\varphi(b_1) \neq \varphi(b_n)$ and thus $1_1 \neq 1_n$. Consequently, the vertices colored with $1$ (or $1'$) induce an independent set.

Given that the \hyperref[alg1]{\texttt{CONFLICTS-RESOLVING}} algorithm does not introduce any new vertices colored with $1$ or $1'$, the conclusion that vertices colored with $1$ (or $1'$) induce an independent set remains valid even after executing line \ref{algo:pc-line13}.
\end{proof}

A \textit{2-independent set} is defined as a vertex subset in which the distance between any pair of vertices is at least 3.

\begin{lemma} \label{lem:2a2b}  
Once the execution of the \hyperref[alg1]{\texttt{CONFLICTS-RESOLVING}} algorithm is finished,
\begin{enumerate}
    \item vertices colored with $2_a$ (or $2_b$) induce a $2$-independent set;
    \item  for every $i \in [n]$, the set $\mathcal{N}[b_i]$ does not contain two vertices colored with the color $2_c$.\label{obs3.3-2}
\end{enumerate}

\end{lemma}

\begin{proof}
Upon completing line \ref{algo:pc-line12} of the \hyperref[alg-final]{\texttt{PACKING-COLORING}} algorithm, the color $2_a$ or $2_b$ is present exclusively within the adjoint cycle $C$.

Now, it is straightforward to observe that if two vertices $a_i$ and $a_j$ are both colored with $2_a$ (or $2_b$) and are located at a distance of at most $2$ from each other, then $b_i = b_j$. In this case, we simply say that vertex $a_i$ has an \textit{AB-conflict} caused by $b_i$.

We now show that the \hyperref[alg1]{\texttt{CONFLICTS-RESOLVING}} algorithm can resolve all AB-conflicts.
Suppose that $a_i$ has an AB-conflict caused by $b_i$. It follows $|\mathcal{N}[b_i]|\geq 2$.

If $|\mathcal{N}[b_i]|=2$, then there is another vertex $a_j$ adjacent to $b_i$ that has the same color with $a_i$. 
The \hyperref[alg1]{\texttt{CONFLICTS-RESOLVING}} algorithm would recolor one of $a_i$ or $a_j$ with the color $2_c$, thus resolving the AB conflict caused by $b_i$.

If $|\mathcal{N}[b_i]|\neq 2$, then $|\mathcal{N}[b_i]|=3$ by Lemma \ref{claim-1}. Let $a_i$, $a_j$, and $a_k$ be the three vertices on $C$ that are adjacent to vertex $b_i$. The proof of Lemma \ref{claim-2} specifies the initial colors assigned to $a_i$, $a_j$, and $a_k$. Upon executing the \hyperref[alg1]{\texttt{CONFLICTS-RESOLVING}} algorithm, exactly one of these vertices ($a_i$, $a_j$, or $a_k$) is recolored. Consequently, after the algorithm completes, the colors of $a_i$, $a_j$, and $a_k$ will consist of $2_a$, $2_b$, and $2_c$. This resolves the AB-conflict caused by $b_i$. 

On the other hand, the \hyperref[alg1]{\texttt{CONFLICTS-RESOLVING}} algorithm would not introduce new AB-conflicts, because $\mathcal{N}[b_i]\cap \mathcal{N}[b_j]=\emptyset$ for every $b_i\neq b_j$.
Therefore, when the \hyperref[alg-final]{\texttt{PACKING-COLORING}} algorithm is applied, every pair of vertices colored with $2_a$ or $2_b$ will be at a distance of at least 3 from each other. This proves the first conclusion.

Next, we can prove the second conclusion in a similar manner. Suppose that, at this stage, there exists some $b_i$ such that $\mathcal{N}[b_i]$ contains two vertices colored with the color $2_c$. Then, there must have been an AB-conflict caused by $b_i$ and $|\mathcal{N}[b_i]| = 3$. However, the previous proof demonstrates that upon applying the \hyperref[alg1]{\texttt{CONFLICTS-RESOLVING}} algorithm, the three vertices in $\mathcal{N}[b_i]$ are respectively colored with $2_a$, $2_b$, and $2_c$. This contradicts our assumption, thus ruling out such a case.
\end{proof}

If the color of a vertex $a_i$ is $2_c$ before applying the \hyperref[alg1]{\texttt{CONFLICTS-RESOLVING}} algorithm, then we refer to $a_i$ as an \textit{original $C$-vertex}. If a vertex $a_i$ is an original $C$-vertex, and $\mathcal{N}[b_i]$ contains either vertices in color $2_a$ or $2_b$, then we say that $a_i$ is \textit{separated}.

If the color of a vertex $a_i$ is not $2_c$ before applying the \hyperref[alg1]{\texttt{CONFLICTS-RESOLVING}} algorithm, but it is recolored to $2_c$ after the algorithm is applied, then we refer to $a_i$ as a \textit{newly introduced $C$-vertex}.
For a newly introduced $C$-vertex $a_i$, it must have an AB-conflict caused by $b_i$ before applying the \hyperref[alg1]{\texttt{CONFLICTS-RESOLVING}} algorithm. Therefore, we let $a_j \in \mathcal{N}[b_i]$ be a vertex such that $a_i$ and $a_j$ have the same color. Note that $j$ is uniquely determined by $i$ according to Lemmas \ref{claim-1} and \ref{claim-2}. For convenience, we set $i^* := j$.

\begin{lem} \label{claim:iteration}
Suppose that \( a_i \) and \( a_j \) are both colored with \( 2_a \) or \( 2_b \) prior to the execution of the \\
\hyperref[alg1]{\texttt{CONFLICTS-RESOLVING}} algorithm, and \( b_i = b_j \) where \( i < j \).  
If \( a_j \) is a newly introduced \( C \)-vertex, then prior to applying the \hyperref[alg1]{\texttt{CONFLICTS-RESOLVING}} algorithm, at least one of the following properties must hold:
\begin{enumerate}[label={\bfseries Property \Roman*}]
    \item Either \( a_{i-2} \) is an original \( C \)-vertex (see Figure \ref{fig:claim}(a)) or a newly introduced \( C \)-vertex with \( (i-2)^* > j \) (see Figure \ref{fig:claim}(b)); or
    \item Property I does not hold, and there exists an original \( C \)-vertex \( a_{k_0} \) and a sequence of newly introduced \( C \)-vertices \( \{a_{k_\ell} : 1 \leq \ell \leq r\} \) along the clockwise path from \( a_n \) to \( a_i \) on \( C \) ($k_0=n$ is possible) such that:
    \begin{itemize}
        \item \( k_r = i-2 \),
        \item \( k_\ell^* = k_{\ell-1} + 2 \) for all \( 1 \leq \ell \leq r \), and
        \item Vertices in \( \{a_{k_\ell + 1}\} \) are colored \( 1 \) or \( 1' \) for every \( 1 \leq \ell \leq r \) (see Figure \ref{fig:claim} (c)); or
    \end{itemize}
    \item Neither I nor II holds, and there exists a sequence of newly introduced \( C \)-vertices \( \{a_{k_\ell} : 0 \leq \ell \leq r\} \) along the clockwise path from \( a_n \) to \( a_i \) on \( C \) such that:
    \begin{itemize}
        \item \( k_0^* > j \),
        \item \( k_r = i-2 \),
        \item \( k_\ell^* = k_{\ell-1} + 2 \) for all \( 1 \leq \ell \leq r \), and
        \item Vertices in \( \{a_{k_{\ell-1} + 1}\} \) are colored \( 1 \) or \( 1' \) for every \( 1 \leq \ell \leq r \) (see Figure \ref{fig:claim} (d)).
    \end{itemize}
    
\end{enumerate}

    \begin{figure}[H]  
    \centering
    \scalebox{1.0}{
        \input{claim}
    }
    \caption{Illustrations of Properties I,II, and III}
    \label{fig:claim}
\end{figure}
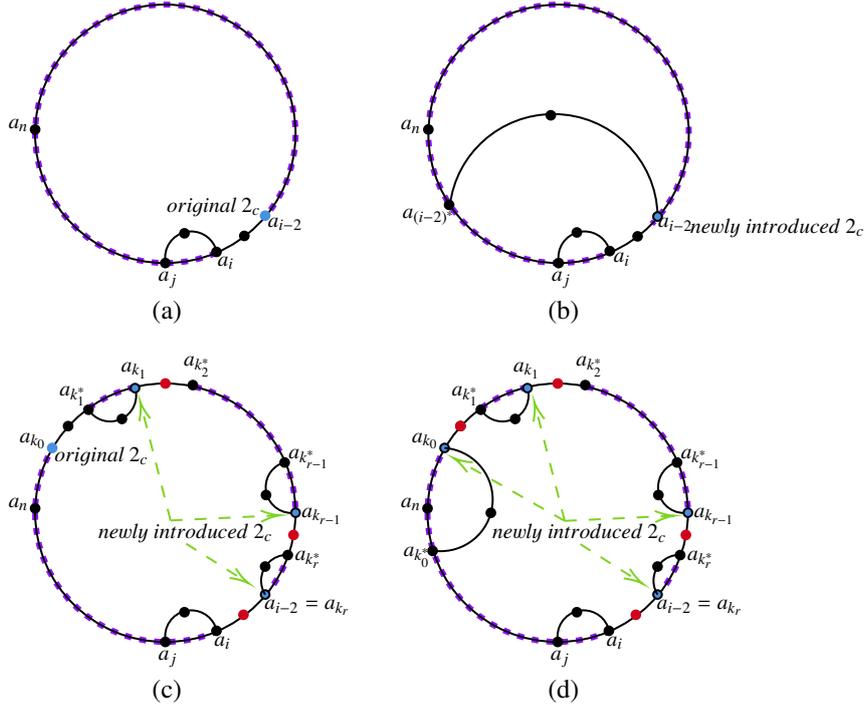

\end{lem}

\begin{proof}
Since \( a_j \) is a newly introduced \( C \)-vertex, we have \( i = j^* \).  
By lines \ref{algo4:line16} and \ref{algo4:line17} of 
\hyperref[alg1]{\texttt{CONFLICTS-RESOLVING}}, when the loop at line \ref{algo4:line9} reaches this specific \( i \), the color of \( a_{i-2} \) must be \( 2_c \).  
If \( a_{i-2} \) is an original \( C \)-vertex or an newly introduced $C$-vertex with $(i-2)^*>j$, the process terminates and we are done.  
Otherwise, \( a_{i-2} \) is a newly introduced \( C \)-vertex with $(i-2)^*<i$. By Lemma \ref{obs3.3-2} the color of \( a_{i-1} \) must be \( 1 \) or \( 1' \).

By replacing \( a_{i-2} \) (i.e., \( a_{j^*-2} \)) with \( a_j \) in the prior analysis, we deduce the existence of another newly introduced \( C \)-vertex encountered earlier along \( C \) (leveraging the property that the Halin graph has no edge crossings).  
Again, by lines \ref{algo4:line16} and \ref{algo4:line17}, when the loop reaches \( (i-2)^* \), the color of \( a_{(i-2)^*-2} \) must be \( 2_c \).  

\begin{itemize}[label=$-$]
    \item If \( a_{(i-2)^*-2} \) is an original \( C \)-vertex or a newly introduced $C$-vertex with $((i-2)^*-2)^*>j$, the process terminates with \( r = 1 \). 
    By Lemma \ref{coloring}, \( a_{(i-2)^*-1} \) has color \( 1 \) or \( 1' \).  
    \item Otherwise, we replace the newly introduced \( C \)-vertex \( a_{(i-2)^*-2} \) with \( a_{i-2} \) and identify \( a_{((i-2)^*-2)^*-2} \). This vertex is either:  
    \begin{itemize}[label=$\bullet$]
        \item an original \( C \)-vertex or a newly introduced $C$-vertex with $(((i-2)^*-2)^*-2)^*>j$ (terminating the process with $r=2$), or  
        \item a newly introduced \( C \)-vertex  with $(((i-2)^*-2)^*-2)^*<j$ (necessitating further iteration).  
    \end{itemize}
\end{itemize}
However, this iteration cannot proceed indefinitely since \( n \) is finite, ensuring termination.
\end{proof}

\begin{lem} \label{lem:2c}
 Once the execution of the \hyperref[alg1]{\texttt{CONFLICTS-RESOLVING}} algorithm is finished,
vertices colored with $2_c$ induce a $2$-independent set.
\end{lem}

\begin{proof}
    We prove this result by contradiction. Suppose that $a_s$ and $a_t$ ($1<s<t$) are colored with $2_c$ by the \hyperref[alg-final]{\texttt{PACKING-COLORING}} algorithm and they are in a distance of at most 2. According to Lemma \hyperref[obs3.3-2]{\texttt{3.3}}, the distance between $a_s$ and $a_t$ along $C$ must be at most 2.

If both $a_s$ and $a_t$ are not newly introduced $C$-vertices, then Lemma \ref{coloring} tells us that we are in the unique case that the color sequence assigned to $a_1, a_2, \ldots, a_n$ is $i2_c i2_b i2_a i2_b \cdots i2_a i2_b i2_a i2_c 2_a$. Thus, the distance between $a_s$ and $a_t$ along $C$ is 3, a contradiction.

Next, we consider three distinct cases for further analysis.

\textbf{Case 1}.
$a_s$ is not a newly introduced $C$-vertex, whereas 
$a_t$ is a newly introduced $C$-vertex.

In this case, we have $s \neq n$. Consequently, according to Lemma \ref{coloring}, the colors assigned to $a_1, a_2, \ldots, a_n$ fall into one of the following three subcases:

\begin{enumerate}[label={(\roman*)}]
    \item \(i2_{c}i2_{b}i2_{a}i2_{b}\cdots i2_{a}i2_{b}i2_{a}i2_{c}2_{a}\);
    \item \(i2_{c}i2_{b}i2_{a}i2_{b}\cdots i2_{a}i2_{b}2_{a}\);
    \item \(i2_{c}i2_{b}i2_{a}i2_{b}\cdots i2_{a}i2_{b}i2_{a}\).
\end{enumerate}

For subcase (i), there are two possibilities: either $s = 2$ and $t = 4$, or $s = n - 1$ and $t = n$. 

In the situation where $s = 2$ and $t = 4$, by referring to the last six lines of the \hyperref[alg1]{\texttt{CONFLICTS-RESOLVING}} algorithm, $a_4$ cannot be a newly introduced $C$-vertex. This is because $\varphi(a_2)=2_c$. 

In the case where $s = n - 1$ and $t = n$, $a_n$ cannot be a newly introduced $C$-vertex. The reason is that $b_n$ is adjacent to both $a_{n - 1}$ and $a_n$, and it is also adjacent to another vertex $a_j$ that is colored with $2_b$. As a result, there is no AB-conflict caused by $b_n$.

For subcase (ii) or (iii),  we have  $s = 2$ and $t = 4$. Similarly, $a_4$ cannot be a newly introduced $C$-vertex because $\varphi(a_2)=2_c$. 

\textbf{Case 2}.
$a_s$ is a newly introduced $C$-vertex, whereas $a_t$ is not a newly introduced $C$-vertex

In this case, we have $t \neq 2$ (because $\varphi(a_1)\not\in \{2_a,2_b,2_c\}$ and thus it cannot be a newly introduced $C$-vertex). Consequently, according to Lemma \ref{coloring}, the colors assigned to $a_1, a_2, \ldots, a_n$ fall into one of the following three subcases:

\begin{enumerate}[label={(\roman*)}]
    \item \(i2_{a}i2_{b}\cdots i2_{a}i2_{b}2_{c}\), where $\mathcal{N}[b_n]$ has a vertex $a_j$ colored with $2_a$ or $2_b$ and $\mathcal{N}[b_n]$ has no other vertex besides $a_j$ colored with $2_b$ when $j=n-1$,  or $b_2$ is not adjacent to an $a_j \neq a_2$ colored with $2_a$;
    \item \(i2_{c}i2_{b}i2_{a}i2_{b}\cdots i2_{a}i2_{b}i2_{a}i2_{c}2_{a}\), where there exists a vertex $b_j$ adjacent to both $a_{n-1}$ and $a_n$, and also adjacent to another vertex $a_j$ colored with $2_b$.
    \item \(i2_{a}i2_{b}\cdots i2_{a}i2_{b}i2_{c}\), where $\mathcal{N}[b_n]$ has no vertex colored with $2_a$ or $2_b$
        and $b_2$ is adjacent to an $a_j \neq a_2$ colored with $2_a$. 
\end{enumerate}

For subcase (i), we have \( s = n-1 \) and \( t = n \). Since \( a_n \) is colored with \( 2_c \), it follows that \( i^* \not> n-1 \) for all \( i \). Additionally, \( a_{n-1} \) is a newly introduced \( C \)-vertex, and
\( a_{n-3} \) is not colored with \( 2_c \).
Thus, we do not satisfy Properties I or III of Lemma \ref{claim:iteration} (with \( j := n-1 \) and \( i := (n-1)^* \)). Instead, we are in Property II of the lemma, which requires the existence of an original \( C \)-vertex \( a_{k_0} \), and
a sequence of newly introduced \( C \)-vertices \( \{a_{k_\ell} : 1 \leq \ell \leq r\} \) along the clockwise path from \( a_n \) to \( a_i \) on \( C \),
such that \( k_r = i-2 \) and \( k_\ell^* = k_{\ell-1} + 2 \) for all \( 1 \leq \ell \leq r \).
 
Since \( a_n \) is the unique original \( C \)-vertex, we must have \( k_0 = n \). The sequence then satisfies \( k_1^* = 2 \) and that \( b_2 \) is adjacent to some vertex other than \( a_2 \) colored with \( 2_a \).
This implies \( \mathcal{N}[b_n] \cap \{2_a, 2_b\} \neq \emptyset \) by the condition of the subcase (i). Suppose \( b_n \) is adjacent to \( a_x \) (\( x \neq n \)) colored with \( 2_a \) or \( 2_b \). Since \( G \) is a Halin graph, it must satisfy \( k_1^* > x \). Given \( k_1^* = 2 \), the only possible \( x \) is \( x = 1 \). However, \( a_1 \) is colored with \( 1 \) or \( 1' \), contradicting the assumption that \( a_x \) is colored with \( 2_a \) or \( 2_b \).

For subcase (ii), we have \(s = n - 3\) and \(t = n - 1\) since \(a_1\) cannot be a newly introduced \(C\)-vertex. In this subcase, \(b_n\) does not cause any AB-conflict, and \(a_{n - 2}\) is colored either with \(1\) or \(1'\). Consequently, \(i^* \ngeq n - 3\) for all \(i\). Thus, following a similar rationale as before, we find ourselves in Property II of Lemma \ref{claim:iteration} (where \(j := n - 3\) and \(i := (n - 3)^*\)). This property necessitates the presence of an original \(C\)-vertex \(a_{k_0}\), along with a sequence of newly introduced \(C\)-vertices \(\{a_{k_\ell} : 1 \leq \ell \leq r\}\) along the clockwise path from \(a_n\) to \(a_i\) on \(C\). This sequence must satisfy \(k_r = i - 2\) and \(k_\ell^* = k_{\ell - 1} + 2\) for all \(1 \leq \ell \leq r\). 

There are only two original \(C\)-vertices. 
If \(k_0 = n\), then \(k_1^* = 2\), which is impossible because \(a_2\) is an original \(C\)-vertex. If \(k_0 = 2\), then \(k_1^* = 4\).
Since \(a_j\) has the color \(2_b\) and \(\mathcal{N}[b_j] = \{a_j, a_{n-1}, a_n\}\), we deduce that \(k_1^* > j\) by the fact that \(G\) is a Halin graph.
It follows that \(j\leq 3\). However, \(a_1\) and \(a_3\) are colored either with \(1\) or \(1'\), and \(a_2\) is colored with \(2_c\). This contradicts the assumption that \(a_j\) is colored with \(2_b\).

For subcase (iii), we can handle it similarly to how we dealt with subcase (i).

\textbf{Case 3}.
$a_s$ and $a_t$ are both newly introduced $C$-vertex.

First, we have $t^* < s$. Otherwise, by lines \ref{algo4:line16}--\ref{algo4:line17} of the \hyperref[alg1]{\texttt{CONFLICTS-RESOLVING}} algorithm, $a_t$ would not be recolored with $2_c$. By the definition of $t^*$, we know that $b_t$ causes an $AB$-conflict and $a_{t^*}$ is colored with $2_a$ or $2_b$. 
By Lemma~\ref{coloring}, regardless of the situation, $a_1$ is colored with $1$ or $1'$, and original $C$-vertices can exist at most among $a_2$, $a_{n-1}$, and $a_n$. Thus, $t^*\geq 2$ and  $a_{t^*}$ is not an original $C$-vertex. 
If there exists an original $C$-vertex $a_j$ with $t^* < j < s$, then $j = 2$. This implies $t^* = 1$, a contradiction.

Now, since each $a_j$ with $t^* < j < s$ is not an original $C$-vertex, it follows that $a_{s-2}$ is also not an original $C$-vertex. Given that $G$ is a Halin graph, we have $t^* < s^* < s$ and $x^* \leq s$ for every $t^* < x < s$. 

Applying Lemma~\ref{claim:iteration} with $j = s$ and $i=s^*$, we observe that its Property~I does not hold. 

If Property~II holds instead, then there exists an original $C$-vertex $a_{k_0}$ and a sequence of newly introduced $C$-vertices $\{a_{k_\ell} : 1 \leq \ell \leq r\}$ along the clockwise path from $a_n$ to $a_i$ on $C$, such that $k_r = i - 2$ and $k_\ell^* = k_{\ell-1} + 2$ for all $1 \leq \ell \leq r$. Since $G$ is a Halin graph, $k_0 + 2 \pmod{n} = k_1^* > t^* \geq 2$. By Lemma~\ref{coloring}, $k_0$ must be one of $\{2, n-1, n\}$. Thus, $k_0 = 2$ and consequently $k_1^* = 4$. This forces $t^* = 3$. However, $a_3$ is always colored with $1$ or $1'$ by Lemma~\ref{coloring}, contradicting the fact that $a_{t^*}$ is colored with $2_a$ or $2_b$.

We now consider the final case where Property~III of Lemma~\ref{claim:iteration} holds. That is, there exists a sequence of newly introduced $C$-vertices $\{a_{k_\ell} : 0 \leq \ell \leq r\}$ along the clockwise path from $a_n$ to $a_i$ on $C$, satisfying:  
$k_0^* > s$, 
$k_r = i - 2$, and  
$k_\ell^* = k_{\ell-1} + 2$ for all $1 \leq \ell \leq r$.  
Since $G$ is a Halin graph, its structural properties imply $t^* < k_0 < s$. However, this directly contradicts the established condition that $x^* \leq s$ for every $t^* < x < s$.
\end{proof}

\begin{theorem}
    The \hyperref[alg-final]{\texttt{PACKING-COLORING}} algorithm returns a $(1^2,2^3)$-packing coloring.
\end{theorem}

\begin{proof}
    This follows directly from Lemmas~\ref{lem:11'}, \ref{lem:2a2b}, and~\ref{lem:2c}.
\end{proof}

\begin{theorem}
  The \hyperref[alg-final]{\texttt{PACKING-COLORING}} algorithm operates in linear time, i.e., it has a time complexity of \(O(|G|)\).
\end{theorem}

\begin{proof}

First, the 2-coloring process applied to the characteristic tree \(T\) using a Depth-First Search (BFS) traversal operates in \(O(t)\) time, where \(t\) represents the number of vertices in \(T\). The \hyperref[alg-recolor]{\texttt{RECOLORING}} algorithm, which encompasses the \hyperref[alg-case-1]{\texttt{CASE-1}} and \hyperref[alg-case-2]{\texttt{CASE-2}} subroutines, processes each case in \(O(n)\) time, with \(n\) being the number of vertices in the adjoint cycle \(C\). The \hyperref[alg1]{\texttt{CONFLICTS-RESOLVING}} algorithm, crucial for resolving color conflicts, also operates within \(O(n)\) time. This is because constructing the set \(\mathcal{L}\) and the dictionary \(\mathcal{N}\) each take \(O(n)\) time. The conflict resolution loop iterates over each even index \(i\) from \(1\) to \(n\). For each even \(i\), the algorithm computes the set \(J\). Given that the size of \(J\) is roughly bounded by 4 (note that $\mathcal{N}[b_j]$ has at most 5 vertices because the maximum degree of $G$ is at most 5), the inner loop iterating over \(J\) runs in constant time \(O(1)\), as it performs at most 4 iterations. Therefore, the entire conflict resolution loop retains a linear time complexity of \(O(n)\), since the outer loop is \(O(n)\) and the inner loop is \(O(1)\).
Consequently, the overall time complexity of the \hyperref[alg-final]{\texttt{PACKING-COLORING}} algorithm is $O(t+n)=O(|G|)$.
\end{proof}

\section{Open questions}\label{sec:4}

When applying the \hyperref[alg-final]{\texttt{PACKING-COLORING}} algorithm to color a Halin graph $G$ with $\Delta(G) \leq 5$, we observe that color $2_c$ is assigned to relatively few vertices. This observation leads us to investigate whether the algorithm can be modified to ensure that the minimum distance between any two vertices colored with $2_c$ exceeds the current constraint. We formalize this inquiry as follows:

\begin{question}
    Are Halin graphs with maximum degree at most 5 $(1^2,2^2,3)$-packing colorable?
\end{question}

In the \hyperref[alg-final]{\texttt{PACKING-COLORING}} algorithm, all non-leaf vertices of the characteristic tree of a Halin graph $G$ are colored $1$ or $1'$. This actually gives a special $(1^2, 2^3)$-packing coloring of $G$. Perhaps the following question is interesting:

\begin{question}
Let $G$ be a Halin graph with maximum degree $\Delta$, where $T$ is the characteristic tree and $C$ is the adjoint cycle. What is the minimum $\ell:=\ell(\Delta)$ such that $G$ admits a \textit{$(1^2, 2^\ell)$-packing coloring} where all vertices in $V(T) \setminus V(C)$ (the internal vertices of $T$ not on $C$) are colored with the first two colors labeled by $1$ in the sequence?
\end{question}

For Halin graphs with maximum degree $\Delta \geq 6$, we pose the following fundamental problem in packing coloring theory:

\begin{question}
    Determine the optimal parameters $(k,\ell)$ such that every Halin graph $G$ with maximum degree $\Delta$ admits a $(1^k,2^\ell)$-packing coloring. Specifically, how does the feasible region of $(k,\ell)$ expand as $\Delta$ increases?
\end{question}

\section*{Acknowledgments}

The authors thank Ping Chen, Mengyao Dai, and Yali Wu for participating in the initial discussion of this paper's topic. In particular, the first author extends gratitude to Weichan Liu for providing guidance on drawing the figures included in this work.

\bibliography{ref}
\bibliographystyle{abbrv}

\end{document}

%% file: claim.tex
\tikzset{every picture/.style={line width=0.75pt}} 

\begin{tikzpicture}[x=0.65pt,y=0.65pt,yscale=-1,xscale=1]

\draw  [draw opacity=0][dash pattern={on 2.53pt off 3.02pt}][line width=2.25]  (350.66,170.5) .. controls (342.65,173.41) and (334.01,175) .. (325,175) .. controls (283.58,175) and (250,141.42) .. (250,100) .. controls (250,58.58) and (283.58,25) .. (325,25) .. controls (366.42,25) and (400,58.58) .. (400,100) .. controls (400,120.36) and (391.89,138.83) .. (378.71,152.34) -- (325,100) -- cycle ; \draw  [color={rgb, 255:red, 144; green, 19; blue, 254 }  ,draw opacity=1 ][dash pattern={on 2.53pt off 3.02pt}][line width=2.25]  (350.66,170.5) .. controls (342.65,173.41) and (334.01,175) .. (325,175) .. controls (283.58,175) and (250,141.42) .. (250,100) .. controls (250,58.58) and (283.58,25) .. (325,25) .. controls (366.42,25) and (400,58.58) .. (400,100) .. controls (400,120.36) and (391.89,138.83) .. (378.71,152.34) ;  
\draw  [draw opacity=0][dash pattern={on 2.53pt off 3.02pt}][line width=2.25]  (127.81,169.06) .. controls (118.81,172.88) and (108.9,175) .. (98.5,175) .. controls (57.08,175) and (23.5,141.42) .. (23.5,100) .. controls (23.5,58.58) and (57.08,25) .. (98.5,25) .. controls (139.92,25) and (173.5,58.58) .. (173.5,100) .. controls (173.5,118.22) and (167,134.92) .. (156.2,147.92) -- (98.5,100) -- cycle ; \draw  [color={rgb, 255:red, 144; green, 19; blue, 254 }  ,draw opacity=1 ][dash pattern={on 2.53pt off 3.02pt}][line width=2.25]  (127.81,169.06) .. controls (118.81,172.88) and (108.9,175) .. (98.5,175) .. controls (57.08,175) and (23.5,141.42) .. (23.5,100) .. controls (23.5,58.58) and (57.08,25) .. (98.5,25) .. controls (139.92,25) and (173.5,58.58) .. (173.5,100) .. controls (173.5,118.22) and (167,134.92) .. (156.2,147.92) ;  
\draw  [draw opacity=0][dash pattern={on 2.53pt off 3.02pt}][line width=2.25]  (54.32,259.29) .. controls (62.36,253.42) and (71.61,249.11) .. (81.62,246.81) -- (98.5,319.9) -- cycle ; \draw  [color={rgb, 255:red, 144; green, 19; blue, 254 }  ,draw opacity=1 ][dash pattern={on 2.53pt off 3.02pt}][line width=2.25]  (54.32,259.29) .. controls (62.36,253.42) and (71.61,249.11) .. (81.62,246.81) ;  
\draw  [draw opacity=0][dash pattern={on 2.53pt off 3.02pt}][line width=2.25]  (129.11,388.39) .. controls (119.76,392.57) and (109.4,394.9) .. (98.5,394.9) .. controls (57.08,394.9) and (23.5,361.32) .. (23.5,319.9) .. controls (23.5,307.38) and (26.57,295.58) .. (31.99,285.2) -- (98.5,319.9) -- cycle ; \draw  [color={rgb, 255:red, 144; green, 19; blue, 254 }  ,draw opacity=1 ][dash pattern={on 2.53pt off 3.02pt}][line width=2.25]  (129.11,388.39) .. controls (119.76,392.57) and (109.4,394.9) .. (98.5,394.9) .. controls (57.08,394.9) and (23.5,361.32) .. (23.5,319.9) .. controls (23.5,307.38) and (26.57,295.58) .. (31.99,285.2) ;  
\draw  [draw opacity=0][dash pattern={on 2.53pt off 3.02pt}][line width=2.25]  (115.46,246.83) .. controls (148.64,254.5) and (173.38,284.17) .. (173.5,319.65) -- (98.5,319.9) -- cycle ; \draw  [color={rgb, 255:red, 144; green, 19; blue, 254 }  ,draw opacity=1 ][dash pattern={on 2.53pt off 3.02pt}][line width=2.25]  (115.46,246.83) .. controls (148.64,254.5) and (173.38,284.17) .. (173.5,319.65) ;  
\draw  [draw opacity=0][dash pattern={on 2.53pt off 3.02pt}][line width=2.25]  (169.96,342.75) .. controls (167.6,350.14) and (164.12,357.03) .. (159.72,363.23) -- (98.5,319.9) -- cycle ; \draw  [color={rgb, 255:red, 144; green, 19; blue, 254 }  ,draw opacity=1 ][dash pattern={on 2.53pt off 3.02pt}][line width=2.25]  (169.96,342.75) .. controls (167.6,350.14) and (164.12,357.03) .. (159.72,363.23) ;  
\draw  [fill={rgb, 255:red, 0; green, 0; blue, 0 }  ,fill opacity=1 ] (101,175) .. controls (101,173.62) and (99.88,172.5) .. (98.5,172.5) .. controls (97.12,172.5) and (96,173.62) .. (96,175) .. controls (96,176.38) and (97.12,177.5) .. (98.5,177.5) .. controls (99.88,177.5) and (101,176.38) .. (101,175) -- cycle ;
\draw   (23.5,100) .. controls (23.5,58.58) and (57.08,25) .. (98.5,25) .. controls (139.92,25) and (173.5,58.58) .. (173.5,100) .. controls (173.5,141.42) and (139.92,175) .. (98.5,175) .. controls (57.08,175) and (23.5,141.42) .. (23.5,100) -- cycle ;
\draw  [fill={rgb, 255:red, 0; green, 0; blue, 0 }  ,fill opacity=1 ] (130.45,168.54) .. controls (130.45,167.16) and (129.33,166.04) .. (127.95,166.04) .. controls (126.57,166.04) and (125.45,167.16) .. (125.45,168.54) .. controls (125.45,169.92) and (126.57,171.04) .. (127.95,171.04) .. controls (129.33,171.04) and (130.45,169.92) .. (130.45,168.54) -- cycle ;
\draw  [fill={rgb, 255:red, 0; green, 0; blue, 0 }  ,fill opacity=1 ] (146.5,159.25) .. controls (146.5,157.87) and (145.38,156.75) .. (144,156.75) .. controls (142.62,156.75) and (141.5,157.87) .. (141.5,159.25) .. controls (141.5,160.63) and (142.62,161.75) .. (144,161.75) .. controls (145.38,161.75) and (146.5,160.63) .. (146.5,159.25) -- cycle ;
\draw  [color={rgb, 255:red, 74; green, 144; blue, 226 }  ,draw opacity=1 ][fill={rgb, 255:red, 74; green, 144; blue, 226 }  ,fill opacity=1 ][line width=0.75]  (158.5,147.5) .. controls (158.5,146.12) and (157.38,145) .. (156,145) .. controls (154.62,145) and (153.5,146.12) .. (153.5,147.5) .. controls (153.5,148.88) and (154.62,150) .. (156,150) .. controls (157.38,150) and (158.5,148.88) .. (158.5,147.5) -- cycle ;
\draw  [fill={rgb, 255:red, 0; green, 0; blue, 0 }  ,fill opacity=1 ] (327.5,175) .. controls (327.5,173.62) and (326.38,172.5) .. (325,172.5) .. controls (323.62,172.5) and (322.5,173.62) .. (322.5,175) .. controls (322.5,176.38) and (323.62,177.5) .. (325,177.5) .. controls (326.38,177.5) and (327.5,176.38) .. (327.5,175) -- cycle ;
\draw   (250,100) .. controls (250,58.58) and (283.58,25) .. (325,25) .. controls (366.42,25) and (400,58.58) .. (400,100) .. controls (400,141.42) and (366.42,175) .. (325,175) .. controls (283.58,175) and (250,141.42) .. (250,100) -- cycle ;
\draw  [fill={rgb, 255:red, 0; green, 0; blue, 0 }  ,fill opacity=1 ] (338,157.25) .. controls (338,155.87) and (336.88,154.75) .. (335.5,154.75) .. controls (334.12,154.75) and (333,155.87) .. (333,157.25) .. controls (333,158.63) and (334.12,159.75) .. (335.5,159.75) .. controls (336.88,159.75) and (338,158.63) .. (338,157.25) -- cycle ;
\draw  [fill={rgb, 255:red, 0; green, 0; blue, 0 }  ,fill opacity=1 ] (356.95,168.04) .. controls (356.95,166.66) and (355.83,165.54) .. (354.45,165.54) .. controls (353.07,165.54) and (351.95,166.66) .. (351.95,168.04) .. controls (351.95,169.42) and (353.07,170.54) .. (354.45,170.54) .. controls (355.83,170.54) and (356.95,169.42) .. (356.95,168.04) -- cycle ;
\draw  [fill={rgb, 255:red, 0; green, 0; blue, 0 }  ,fill opacity=1 ] (373,159.5) .. controls (373,158.12) and (371.88,157) .. (370.5,157) .. controls (369.12,157) and (368,158.12) .. (368,159.5) .. controls (368,160.88) and (369.12,162) .. (370.5,162) .. controls (371.88,162) and (373,160.88) .. (373,159.5) -- cycle ;
\draw  [color={rgb, 255:red, 74; green, 144; blue, 226 }  ,draw opacity=1 ][fill={rgb, 255:red, 74; green, 144; blue, 226 }  ,fill opacity=1 ] (384.5,148) .. controls (384.5,146.62) and (383.38,145.5) .. (382,145.5) .. controls (380.62,145.5) and (379.5,146.62) .. (379.5,148) .. controls (379.5,149.38) and (380.62,150.5) .. (382,150.5) .. controls (383.38,150.5) and (384.5,149.38) .. (384.5,148) -- cycle ;
\draw  [fill={rgb, 255:red, 0; green, 0; blue, 0 }  ,fill opacity=1 ] (323,89.25) .. controls (323,87.87) and (321.88,86.75) .. (320.5,86.75) .. controls (319.12,86.75) and (318,87.87) .. (318,89.25) .. controls (318,90.63) and (319.12,91.75) .. (320.5,91.75) .. controls (321.88,91.75) and (323,90.63) .. (323,89.25) -- cycle ;
\draw  [fill={rgb, 255:red, 0; green, 0; blue, 0 }  ,fill opacity=1 ] (252.5,97.5) .. controls (252.5,96.12) and (251.38,95) .. (250,95) .. controls (248.62,95) and (247.5,96.12) .. (247.5,97.5) .. controls (247.5,98.88) and (248.62,100) .. (250,100) .. controls (251.38,100) and (252.5,98.88) .. (252.5,97.5) -- cycle ;
\draw  [fill={rgb, 255:red, 0; green, 0; blue, 0 }  ,fill opacity=1 ] (264.45,141.23) .. controls (264.45,139.85) and (263.33,138.73) .. (261.95,138.73) .. controls (260.57,138.73) and (259.45,139.85) .. (259.45,141.23) .. controls (259.45,142.61) and (260.57,143.73) .. (261.95,143.73) .. controls (263.33,143.73) and (264.45,142.61) .. (264.45,141.23) -- cycle ;
\draw  [fill={rgb, 255:red, 0; green, 0; blue, 0 }  ,fill opacity=1 ] (45,269.65) .. controls (45,268.27) and (43.88,267.15) .. (42.5,267.15) .. controls (41.12,267.15) and (40,268.27) .. (40,269.65) .. controls (40,271.03) and (41.12,272.15) .. (42.5,272.15) .. controls (43.88,272.15) and (45,271.03) .. (45,269.65) -- cycle ;
\draw  [fill={rgb, 255:red, 0; green, 0; blue, 0 }  ,fill opacity=1 ] (26,97.5) .. controls (26,96.12) and (24.88,95) .. (23.5,95) .. controls (22.12,95) and (21,96.12) .. (21,97.5) .. controls (21,98.88) and (22.12,100) .. (23.5,100) .. controls (24.88,100) and (26,98.88) .. (26,97.5) -- cycle ;
\draw  [draw opacity=0] (98.5,173.5) .. controls (98.43,172.9) and (98.39,172.29) .. (98.39,171.67) .. controls (98.39,163.39) and (105.11,156.67) .. (113.39,156.67) .. controls (120.42,156.67) and (126.32,161.51) .. (127.95,168.04) -- (113.39,171.67) -- cycle ; \draw   (98.5,173.5) .. controls (98.43,172.9) and (98.39,172.29) .. (98.39,171.67) .. controls (98.39,163.39) and (105.11,156.67) .. (113.39,156.67) .. controls (120.42,156.67) and (126.32,161.51) .. (127.95,168.04) ;  
\draw  [fill={rgb, 255:red, 0; green, 0; blue, 0 }  ,fill opacity=1 ] (112,157.75) .. controls (112,156.37) and (110.88,155.25) .. (109.5,155.25) .. controls (108.12,155.25) and (107,156.37) .. (107,157.75) .. controls (107,159.13) and (108.12,160.25) .. (109.5,160.25) .. controls (110.88,160.25) and (112,159.13) .. (112,157.75) -- cycle ;
\draw  [draw opacity=0] (325,173.5) .. controls (324.93,172.9) and (324.89,172.29) .. (324.89,171.67) .. controls (324.89,163.39) and (331.61,156.67) .. (339.89,156.67) .. controls (346.92,156.67) and (352.82,161.51) .. (354.45,168.04) -- (339.89,171.67) -- cycle ; \draw   (325,173.5) .. controls (324.93,172.9) and (324.89,172.29) .. (324.89,171.67) .. controls (324.89,163.39) and (331.61,156.67) .. (339.89,156.67) .. controls (346.92,156.67) and (352.82,161.51) .. (354.45,168.04) ;  
\draw  [draw opacity=0] (262.45,141.23) .. controls (266.06,111.55) and (291.35,88.55) .. (322.01,88.55) .. controls (354.8,88.55) and (381.44,114.85) .. (382,147.5) -- (322.01,148.55) -- cycle ; \draw   (262.45,141.23) .. controls (266.06,111.55) and (291.35,88.55) .. (322.01,88.55) .. controls (354.8,88.55) and (381.44,114.85) .. (382,147.5) ;  
\draw  [fill={rgb, 255:red, 0; green, 0; blue, 0 }  ,fill opacity=1 ] (101,394.9) .. controls (101,393.52) and (99.88,392.4) .. (98.5,392.4) .. controls (97.12,392.4) and (96,393.52) .. (96,394.9) .. controls (96,396.28) and (97.12,397.4) .. (98.5,397.4) .. controls (99.88,397.4) and (101,396.28) .. (101,394.9) -- cycle ;
\draw   (23.5,319.9) .. controls (23.5,278.48) and (57.08,244.9) .. (98.5,244.9) .. controls (139.92,244.9) and (173.5,278.48) .. (173.5,319.9) .. controls (173.5,361.32) and (139.92,394.9) .. (98.5,394.9) .. controls (57.08,394.9) and (23.5,361.32) .. (23.5,319.9) -- cycle ;
\draw  [fill={rgb, 255:red, 0; green, 0; blue, 0 }  ,fill opacity=1 ] (130.45,388.44) .. controls (130.45,387.06) and (129.33,385.94) .. (127.95,385.94) .. controls (126.57,385.94) and (125.45,387.06) .. (125.45,388.44) .. controls (125.45,389.82) and (126.57,390.94) .. (127.95,390.94) .. controls (129.33,390.94) and (130.45,389.82) .. (130.45,388.44) -- cycle ;
\draw  [color={rgb, 255:red, 208; green, 2; blue, 27 }  ,draw opacity=1 ][fill={rgb, 255:red, 208; green, 2; blue, 27 }  ,fill opacity=1 ] (146,379.15) .. controls (146,377.77) and (144.88,376.65) .. (143.5,376.65) .. controls (142.12,376.65) and (141,377.77) .. (141,379.15) .. controls (141,380.53) and (142.12,381.65) .. (143.5,381.65) .. controls (144.88,381.65) and (146,380.53) .. (146,379.15) -- cycle ;
\draw  [color={rgb, 255:red, 0; green, 0; blue, 0 }  ,draw opacity=1 ][fill={rgb, 255:red, 74; green, 144; blue, 226 }  ,fill opacity=1 ][line width=0.75]  (158.5,367.4) .. controls (158.5,366.02) and (157.38,364.9) .. (156,364.9) .. controls (154.62,364.9) and (153.5,366.02) .. (153.5,367.4) .. controls (153.5,368.78) and (154.62,369.9) .. (156,369.9) .. controls (157.38,369.9) and (158.5,368.78) .. (158.5,367.4) -- cycle ;
\draw  [fill={rgb, 255:red, 0; green, 0; blue, 0 }  ,fill opacity=1 ] (26,317.4) .. controls (26,316.02) and (24.88,314.9) .. (23.5,314.9) .. controls (22.12,314.9) and (21,316.02) .. (21,317.4) .. controls (21,318.78) and (22.12,319.9) .. (23.5,319.9) .. controls (24.88,319.9) and (26,318.78) .. (26,317.4) -- cycle ;
\draw  [draw opacity=0] (98.5,393.4) .. controls (98.43,392.8) and (98.39,392.19) .. (98.39,391.57) .. controls (98.39,383.29) and (105.11,376.57) .. (113.39,376.57) .. controls (120.42,376.57) and (126.32,381.41) .. (127.95,387.94) -- (113.39,391.57) -- cycle ; \draw   (98.5,393.4) .. controls (98.43,392.8) and (98.39,392.19) .. (98.39,391.57) .. controls (98.39,383.29) and (105.11,376.57) .. (113.39,376.57) .. controls (120.42,376.57) and (126.32,381.41) .. (127.95,387.94) ;  
\draw  [fill={rgb, 255:red, 0; green, 0; blue, 0 }  ,fill opacity=1 ] (112,377.65) .. controls (112,376.27) and (110.88,375.15) .. (109.5,375.15) .. controls (108.12,375.15) and (107,376.27) .. (107,377.65) .. controls (107,379.03) and (108.12,380.15) .. (109.5,380.15) .. controls (110.88,380.15) and (112,379.03) .. (112,377.65) -- cycle ;
\draw  [draw opacity=0] (156,367.4) .. controls (154.53,365.08) and (153.68,362.32) .. (153.68,359.37) .. controls (153.68,351.08) and (160.37,344.37) .. (168.62,344.37) .. controls (169.27,344.37) and (169.9,344.41) .. (170.52,344.49) -- (168.62,359.37) -- cycle ; \draw   (156,367.4) .. controls (154.53,365.08) and (153.68,362.32) .. (153.68,359.37) .. controls (153.68,351.08) and (160.37,344.37) .. (168.62,344.37) .. controls (169.27,344.37) and (169.9,344.41) .. (170.52,344.49) ;  
\draw  [draw opacity=0] (173.5,319.9) .. controls (172.77,320.01) and (172.03,320.07) .. (171.27,320.07) .. controls (163.02,320.07) and (156.33,313.35) .. (156.33,305.07) .. controls (156.33,298.38) and (160.69,292.71) .. (166.72,290.77) -- (171.27,305.07) -- cycle ; \draw   (173.5,319.9) .. controls (172.77,320.01) and (172.03,320.07) .. (171.27,320.07) .. controls (163.02,320.07) and (156.33,313.35) .. (156.33,305.07) .. controls (156.33,298.38) and (160.69,292.71) .. (166.72,290.77) ;  
\draw  [draw opacity=0] (81.05,247.59) .. controls (81.47,248.97) and (81.69,250.43) .. (81.69,251.95) .. controls (81.69,260.23) and (75,266.95) .. (66.75,266.95) .. controls (61.51,266.95) and (56.9,264.24) .. (54.23,260.15) -- (66.75,251.95) -- cycle ; \draw   (81.05,247.59) .. controls (81.47,248.97) and (81.69,250.43) .. (81.69,251.95) .. controls (81.69,260.23) and (75,266.95) .. (66.75,266.95) .. controls (61.51,266.95) and (56.9,264.24) .. (54.23,260.15) ;  
\draw  [fill={rgb, 255:red, 0; green, 0; blue, 0 }  ,fill opacity=1 ] (171.52,344.49) .. controls (171.52,343.11) and (170.4,341.99) .. (169.02,341.99) .. controls (167.64,341.99) and (166.52,343.11) .. (166.52,344.49) .. controls (166.52,345.87) and (167.64,346.99) .. (169.02,346.99) .. controls (170.4,346.99) and (171.52,345.87) .. (171.52,344.49) -- cycle ;
\draw  [fill={rgb, 255:red, 0; green, 0; blue, 0 }  ,fill opacity=1 ] (158.5,351.15) .. controls (158.5,349.77) and (157.38,348.65) .. (156,348.65) .. controls (154.62,348.65) and (153.5,349.77) .. (153.5,351.15) .. controls (153.5,352.53) and (154.62,353.65) .. (156,353.65) .. controls (157.38,353.65) and (158.5,352.53) .. (158.5,351.15) -- cycle ;
\draw  [fill={rgb, 255:red, 74; green, 144; blue, 226 }  ,fill opacity=1 ] (176,319.9) .. controls (176,318.52) and (174.88,317.4) .. (173.5,317.4) .. controls (172.12,317.4) and (171,318.52) .. (171,319.9) .. controls (171,321.28) and (172.12,322.4) .. (173.5,322.4) .. controls (174.88,322.4) and (176,321.28) .. (176,319.9) -- cycle ;
\draw  [color={rgb, 255:red, 208; green, 2; blue, 27 }  ,draw opacity=1 ][fill={rgb, 255:red, 208; green, 2; blue, 27 }  ,fill opacity=1 ] (174.83,332.83) .. controls (174.83,331.45) and (173.71,330.33) .. (172.33,330.33) .. controls (170.95,330.33) and (169.83,331.45) .. (169.83,332.83) .. controls (169.83,334.21) and (170.95,335.33) .. (172.33,335.33) .. controls (173.71,335.33) and (174.83,334.21) .. (174.83,332.83) -- cycle ;
\draw  [fill={rgb, 255:red, 0; green, 0; blue, 0 }  ,fill opacity=1 ] (159,309.65) .. controls (159,308.27) and (157.88,307.15) .. (156.5,307.15) .. controls (155.12,307.15) and (154,308.27) .. (154,309.65) .. controls (154,311.03) and (155.12,312.15) .. (156.5,312.15) .. controls (157.88,312.15) and (159,311.03) .. (159,309.65) -- cycle ;
\draw  [fill={rgb, 255:red, 0; green, 0; blue, 0 }  ,fill opacity=1 ] (169.72,290.77) .. controls (169.72,289.39) and (168.6,288.27) .. (167.22,288.27) .. controls (165.84,288.27) and (164.72,289.39) .. (164.72,290.77) .. controls (164.72,292.16) and (165.84,293.27) .. (167.22,293.27) .. controls (168.6,293.27) and (169.72,292.16) .. (169.72,290.77) -- cycle ;
\draw  [fill={rgb, 255:red, 74; green, 144; blue, 226 }  ,fill opacity=1 ] (83.55,247.59) .. controls (83.55,246.21) and (82.43,245.09) .. (81.05,245.09) .. controls (79.67,245.09) and (78.55,246.21) .. (78.55,247.59) .. controls (78.55,248.97) and (79.67,250.09) .. (81.05,250.09) .. controls (82.43,250.09) and (83.55,248.97) .. (83.55,247.59) -- cycle ;
\draw  [color={rgb, 255:red, 208; green, 2; blue, 27 }  ,draw opacity=1 ][fill={rgb, 255:red, 208; green, 2; blue, 27 }  ,fill opacity=1 ] (101,244.9) .. controls (101,243.52) and (99.88,242.4) .. (98.5,242.4) .. controls (97.12,242.4) and (96,243.52) .. (96,244.9) .. controls (96,246.28) and (97.12,247.4) .. (98.5,247.4) .. controls (99.88,247.4) and (101,246.28) .. (101,244.9) -- cycle ;
\draw  [fill={rgb, 255:red, 0; green, 0; blue, 0 }  ,fill opacity=1 ] (56.73,260.15) .. controls (56.73,258.77) and (55.61,257.65) .. (54.23,257.65) .. controls (52.85,257.65) and (51.73,258.77) .. (51.73,260.15) .. controls (51.73,261.53) and (52.85,262.65) .. (54.23,262.65) .. controls (55.61,262.65) and (56.73,261.53) .. (56.73,260.15) -- cycle ;
\draw  [color={rgb, 255:red, 74; green, 144; blue, 226 }  ,draw opacity=1 ][fill={rgb, 255:red, 74; green, 144; blue, 226 }  ,fill opacity=1 ] (35.83,282.33) .. controls (35.83,280.95) and (34.71,279.83) .. (33.33,279.83) .. controls (31.95,279.83) and (30.83,280.95) .. (30.83,282.33) .. controls (30.83,283.71) and (31.95,284.83) .. (33.33,284.83) .. controls (34.71,284.83) and (35.83,283.71) .. (35.83,282.33) -- cycle ;
\draw  [color={rgb, 255:red, 0; green, 0; blue, 0 }  ,draw opacity=1 ][fill={rgb, 255:red, 0; green, 0; blue, 0 }  ,fill opacity=1 ] (116.83,245.83) .. controls (116.83,244.45) and (115.71,243.33) .. (114.33,243.33) .. controls (112.95,243.33) and (111.83,244.45) .. (111.83,245.83) .. controls (111.83,247.21) and (112.95,248.33) .. (114.33,248.33) .. controls (115.71,248.33) and (116.83,247.21) .. (116.83,245.83) -- cycle ;
\draw  [draw opacity=0][dash pattern={on 2.53pt off 3.02pt}][line width=2.25]  (280.32,259.29) .. controls (288.36,253.42) and (297.61,249.11) .. (307.62,246.8) -- (324.5,319.9) -- cycle ; \draw  [color={rgb, 255:red, 144; green, 19; blue, 254 }  ,draw opacity=1 ][dash pattern={on 2.53pt off 3.02pt}][line width=2.25]  (280.32,259.29) .. controls (288.36,253.42) and (297.61,249.11) .. (307.62,246.8) ;  
\draw  [draw opacity=0][dash pattern={on 2.53pt off 3.02pt}][line width=2.25]  (355.11,388.39) .. controls (345.76,392.57) and (335.4,394.9) .. (324.5,394.9) .. controls (283.08,394.9) and (249.5,361.32) .. (249.5,319.9) .. controls (249.5,307.38) and (252.57,295.58) .. (257.99,285.2) -- (324.5,319.9) -- cycle ; \draw  [color={rgb, 255:red, 144; green, 19; blue, 254 }  ,draw opacity=1 ][dash pattern={on 2.53pt off 3.02pt}][line width=2.25]  (355.11,388.39) .. controls (345.76,392.57) and (335.4,394.9) .. (324.5,394.9) .. controls (283.08,394.9) and (249.5,361.32) .. (249.5,319.9) .. controls (249.5,307.38) and (252.57,295.58) .. (257.99,285.2) ;  
\draw  [draw opacity=0][dash pattern={on 2.53pt off 3.02pt}][line width=2.25]  (341.46,246.83) .. controls (374.64,254.5) and (399.38,284.17) .. (399.5,319.65) -- (324.5,319.9) -- cycle ; \draw  [color={rgb, 255:red, 144; green, 19; blue, 254 }  ,draw opacity=1 ][dash pattern={on 2.53pt off 3.02pt}][line width=2.25]  (341.46,246.83) .. controls (374.64,254.5) and (399.38,284.17) .. (399.5,319.65) ;  
\draw  [draw opacity=0][dash pattern={on 2.53pt off 3.02pt}][line width=2.25]  (395.96,342.74) .. controls (393.6,350.14) and (390.12,357.03) .. (385.72,363.23) -- (324.5,319.9) -- cycle ; \draw  [color={rgb, 255:red, 144; green, 19; blue, 254 }  ,draw opacity=1 ][dash pattern={on 2.53pt off 3.02pt}][line width=2.25]  (395.96,342.74) .. controls (393.6,350.14) and (390.12,357.03) .. (385.72,363.23) ;  
\draw  [fill={rgb, 255:red, 0; green, 0; blue, 0 }  ,fill opacity=1 ] (271,269.65) .. controls (271,268.27) and (269.88,267.15) .. (268.5,267.15) .. controls (267.12,267.15) and (266,268.27) .. (266,269.65) .. controls (266,271.03) and (267.12,272.15) .. (268.5,272.15) .. controls (269.88,272.15) and (271,271.03) .. (271,269.65) -- cycle ;
\draw  [fill={rgb, 255:red, 0; green, 0; blue, 0 }  ,fill opacity=1 ] (327,394.9) .. controls (327,393.52) and (325.88,392.4) .. (324.5,392.4) .. controls (323.12,392.4) and (322,393.52) .. (322,394.9) .. controls (322,396.28) and (323.12,397.4) .. (324.5,397.4) .. controls (325.88,397.4) and (327,396.28) .. (327,394.9) -- cycle ;
\draw   (249.5,319.9) .. controls (249.5,278.48) and (283.08,244.9) .. (324.5,244.9) .. controls (365.92,244.9) and (399.5,278.48) .. (399.5,319.9) .. controls (399.5,361.32) and (365.92,394.9) .. (324.5,394.9) .. controls (283.08,394.9) and (249.5,361.32) .. (249.5,319.9) -- cycle ;
\draw  [fill={rgb, 255:red, 0; green, 0; blue, 0 }  ,fill opacity=1 ] (356.45,388.44) .. controls (356.45,387.06) and (355.33,385.94) .. (353.95,385.94) .. controls (352.57,385.94) and (351.45,387.06) .. (351.45,388.44) .. controls (351.45,389.82) and (352.57,390.94) .. (353.95,390.94) .. controls (355.33,390.94) and (356.45,389.82) .. (356.45,388.44) -- cycle ;
\draw  [color={rgb, 255:red, 208; green, 2; blue, 27 }  ,draw opacity=1 ][fill={rgb, 255:red, 208; green, 2; blue, 27 }  ,fill opacity=1 ] (372,379.15) .. controls (372,377.77) and (370.88,376.65) .. (369.5,376.65) .. controls (368.12,376.65) and (367,377.77) .. (367,379.15) .. controls (367,380.53) and (368.12,381.65) .. (369.5,381.65) .. controls (370.88,381.65) and (372,380.53) .. (372,379.15) -- cycle ;
\draw  [color={rgb, 255:red, 0; green, 0; blue, 0 }  ,draw opacity=1 ][fill={rgb, 255:red, 74; green, 144; blue, 226 }  ,fill opacity=1 ][line width=0.75]  (384.5,367.4) .. controls (384.5,366.02) and (383.38,364.9) .. (382,364.9) .. controls (380.62,364.9) and (379.5,366.02) .. (379.5,367.4) .. controls (379.5,368.78) and (380.62,369.9) .. (382,369.9) .. controls (383.38,369.9) and (384.5,368.78) .. (384.5,367.4) -- cycle ;
\draw  [fill={rgb, 255:red, 0; green, 0; blue, 0 }  ,fill opacity=1 ] (252,317.4) .. controls (252,316.02) and (250.88,314.9) .. (249.5,314.9) .. controls (248.12,314.9) and (247,316.02) .. (247,317.4) .. controls (247,318.78) and (248.12,319.9) .. (249.5,319.9) .. controls (250.88,319.9) and (252,318.78) .. (252,317.4) -- cycle ;
\draw  [draw opacity=0] (324.5,393.4) .. controls (324.43,392.8) and (324.39,392.19) .. (324.39,391.57) .. controls (324.39,383.29) and (331.11,376.57) .. (339.39,376.57) .. controls (346.42,376.57) and (352.32,381.41) .. (353.95,387.94) -- (339.39,391.57) -- cycle ; \draw   (324.5,393.4) .. controls (324.43,392.8) and (324.39,392.19) .. (324.39,391.57) .. controls (324.39,383.29) and (331.11,376.57) .. (339.39,376.57) .. controls (346.42,376.57) and (352.32,381.41) .. (353.95,387.94) ;  
\draw  [fill={rgb, 255:red, 0; green, 0; blue, 0 }  ,fill opacity=1 ] (338,377.65) .. controls (338,376.27) and (336.88,375.15) .. (335.5,375.15) .. controls (334.12,375.15) and (333,376.27) .. (333,377.65) .. controls (333,379.03) and (334.12,380.15) .. (335.5,380.15) .. controls (336.88,380.15) and (338,379.03) .. (338,377.65) -- cycle ;
\draw  [draw opacity=0] (382,367.4) .. controls (380.53,365.08) and (379.68,362.32) .. (379.68,359.37) .. controls (379.68,351.08) and (386.37,344.37) .. (394.62,344.37) .. controls (395.27,344.37) and (395.9,344.41) .. (396.52,344.49) -- (394.62,359.37) -- cycle ; \draw   (382,367.4) .. controls (380.53,365.08) and (379.68,362.32) .. (379.68,359.37) .. controls (379.68,351.08) and (386.37,344.37) .. (394.62,344.37) .. controls (395.27,344.37) and (395.9,344.41) .. (396.52,344.49) ;  
\draw  [draw opacity=0] (399.5,319.9) .. controls (398.77,320.01) and (398.03,320.06) .. (397.27,320.06) .. controls (389.02,320.06) and (382.33,313.35) .. (382.33,305.06) .. controls (382.33,298.37) and (386.69,292.71) .. (392.72,290.77) -- (397.27,305.06) -- cycle ; \draw   (399.5,319.9) .. controls (398.77,320.01) and (398.03,320.06) .. (397.27,320.06) .. controls (389.02,320.06) and (382.33,313.35) .. (382.33,305.06) .. controls (382.33,298.37) and (386.69,292.71) .. (392.72,290.77) ;  
\draw  [draw opacity=0] (307.05,247.59) .. controls (307.47,248.97) and (307.69,250.43) .. (307.69,251.95) .. controls (307.69,260.23) and (301,266.95) .. (292.75,266.95) .. controls (287.51,266.95) and (282.9,264.24) .. (280.23,260.15) -- (292.75,251.95) -- cycle ; \draw   (307.05,247.59) .. controls (307.47,248.97) and (307.69,250.43) .. (307.69,251.95) .. controls (307.69,260.23) and (301,266.95) .. (292.75,266.95) .. controls (287.51,266.95) and (282.9,264.24) .. (280.23,260.15) ;  
\draw  [fill={rgb, 255:red, 0; green, 0; blue, 0 }  ,fill opacity=1 ] (397.52,344.49) .. controls (397.52,343.1) and (396.4,341.99) .. (395.02,341.99) .. controls (393.64,341.99) and (392.52,343.1) .. (392.52,344.49) .. controls (392.52,345.87) and (393.64,346.99) .. (395.02,346.99) .. controls (396.4,346.99) and (397.52,345.87) .. (397.52,344.49) -- cycle ;
\draw  [fill={rgb, 255:red, 0; green, 0; blue, 0 }  ,fill opacity=1 ] (384.5,351.15) .. controls (384.5,349.77) and (383.38,348.65) .. (382,348.65) .. controls (380.62,348.65) and (379.5,349.77) .. (379.5,351.15) .. controls (379.5,352.53) and (380.62,353.65) .. (382,353.65) .. controls (383.38,353.65) and (384.5,352.53) .. (384.5,351.15) -- cycle ;
\draw  [fill={rgb, 255:red, 74; green, 144; blue, 226 }  ,fill opacity=1 ] (402,319.9) .. controls (402,318.52) and (400.88,317.4) .. (399.5,317.4) .. controls (398.12,317.4) and (397,318.52) .. (397,319.9) .. controls (397,321.28) and (398.12,322.4) .. (399.5,322.4) .. controls (400.88,322.4) and (402,321.28) .. (402,319.9) -- cycle ;
\draw  [color={rgb, 255:red, 208; green, 2; blue, 27 }  ,draw opacity=1 ][fill={rgb, 255:red, 208; green, 2; blue, 27 }  ,fill opacity=1 ] (400.83,332.83) .. controls (400.83,331.45) and (399.71,330.33) .. (398.33,330.33) .. controls (396.95,330.33) and (395.83,331.45) .. (395.83,332.83) .. controls (395.83,334.21) and (396.95,335.33) .. (398.33,335.33) .. controls (399.71,335.33) and (400.83,334.21) .. (400.83,332.83) -- cycle ;
\draw  [fill={rgb, 255:red, 0; green, 0; blue, 0 }  ,fill opacity=1 ] (385,309.65) .. controls (385,308.27) and (383.88,307.15) .. (382.5,307.15) .. controls (381.12,307.15) and (380,308.27) .. (380,309.65) .. controls (380,311.03) and (381.12,312.15) .. (382.5,312.15) .. controls (383.88,312.15) and (385,311.03) .. (385,309.65) -- cycle ;
\draw  [fill={rgb, 255:red, 0; green, 0; blue, 0 }  ,fill opacity=1 ] (395.72,290.77) .. controls (395.72,289.39) and (394.6,288.27) .. (393.22,288.27) .. controls (391.84,288.27) and (390.72,289.39) .. (390.72,290.77) .. controls (390.72,292.15) and (391.84,293.27) .. (393.22,293.27) .. controls (394.6,293.27) and (395.72,292.15) .. (395.72,290.77) -- cycle ;
\draw  [fill={rgb, 255:red, 74; green, 144; blue, 226 }  ,fill opacity=1 ] (309.55,247.59) .. controls (309.55,246.21) and (308.43,245.09) .. (307.05,245.09) .. controls (305.67,245.09) and (304.55,246.21) .. (304.55,247.59) .. controls (304.55,248.97) and (305.67,250.09) .. (307.05,250.09) .. controls (308.43,250.09) and (309.55,248.97) .. (309.55,247.59) -- cycle ;
\draw  [color={rgb, 255:red, 208; green, 2; blue, 27 }  ,draw opacity=1 ][fill={rgb, 255:red, 208; green, 2; blue, 27 }  ,fill opacity=1 ] (327,244.9) .. controls (327,243.52) and (325.88,242.4) .. (324.5,242.4) .. controls (323.12,242.4) and (322,243.52) .. (322,244.9) .. controls (322,246.28) and (323.12,247.4) .. (324.5,247.4) .. controls (325.88,247.4) and (327,246.28) .. (327,244.9) -- cycle ;
\draw  [fill={rgb, 255:red, 0; green, 0; blue, 0 }  ,fill opacity=1 ] (282.73,260.15) .. controls (282.73,258.77) and (281.61,257.65) .. (280.23,257.65) .. controls (278.85,257.65) and (277.73,258.77) .. (277.73,260.15) .. controls (277.73,261.53) and (278.85,262.65) .. (280.23,262.65) .. controls (281.61,262.65) and (282.73,261.53) .. (282.73,260.15) -- cycle ;
\draw  [color={rgb, 255:red, 0; green, 0; blue, 0 }  ,draw opacity=1 ][fill={rgb, 255:red, 74; green, 144; blue, 226 }  ,fill opacity=1 ] (261.83,282.33) .. controls (261.83,280.95) and (260.71,279.83) .. (259.33,279.83) .. controls (257.95,279.83) and (256.83,280.95) .. (256.83,282.33) .. controls (256.83,283.71) and (257.95,284.83) .. (259.33,284.83) .. controls (260.71,284.83) and (261.83,283.71) .. (261.83,282.33) -- cycle ;
\draw  [color={rgb, 255:red, 0; green, 0; blue, 0 }  ,draw opacity=1 ][fill={rgb, 255:red, 0; green, 0; blue, 0 }  ,fill opacity=1 ] (342.83,245.83) .. controls (342.83,244.45) and (341.71,243.33) .. (340.33,243.33) .. controls (338.95,243.33) and (337.83,244.45) .. (337.83,245.83) .. controls (337.83,247.21) and (338.95,248.33) .. (340.33,248.33) .. controls (341.71,248.33) and (342.83,247.21) .. (342.83,245.83) -- cycle ;
\draw  [color={rgb, 255:red, 208; green, 2; blue, 27 }  ,draw opacity=1 ][fill={rgb, 255:red, 208; green, 2; blue, 27 }  ,fill opacity=1 ] (271,269.65) .. controls (271,268.27) and (269.88,267.15) .. (268.5,267.15) .. controls (267.12,267.15) and (266,268.27) .. (266,269.65) .. controls (266,271.03) and (267.12,272.15) .. (268.5,272.15) .. controls (269.88,272.15) and (271,271.03) .. (271,269.65) -- cycle ;
\draw  [draw opacity=0] (259.33,282.33) .. controls (274.83,283.51) and (287.03,296.45) .. (287.03,312.24) .. controls (287.03,328.81) and (273.6,342.24) .. (257.03,342.24) .. controls (256,342.24) and (254.97,342.19) .. (253.96,342.09) -- (257.03,312.24) -- cycle ; \draw   (259.33,282.33) .. controls (274.83,283.51) and (287.03,296.45) .. (287.03,312.24) .. controls (287.03,328.81) and (273.6,342.24) .. (257.03,342.24) .. controls (256,342.24) and (254.97,342.19) .. (253.96,342.09) ;  
\draw  [fill={rgb, 255:red, 0; green, 0; blue, 0 }  ,fill opacity=1 ] (254.96,342.09) .. controls (254.96,340.71) and (253.84,339.59) .. (252.46,339.59) .. controls (251.08,339.59) and (249.96,340.71) .. (249.96,342.09) .. controls (249.96,343.47) and (251.08,344.59) .. (252.46,344.59) .. controls (253.84,344.59) and (254.96,343.47) .. (254.96,342.09) -- cycle ;
\draw  [fill={rgb, 255:red, 74; green, 144; blue, 226 }  ,fill opacity=1 ] (384.5,148) .. controls (384.5,146.62) and (383.38,145.5) .. (382,145.5) .. controls (380.62,145.5) and (379.5,146.62) .. (379.5,148) .. controls (379.5,149.38) and (380.62,150.5) .. (382,150.5) .. controls (383.38,150.5) and (384.5,149.38) .. (384.5,148) -- cycle ;
\draw [color={rgb, 255:red, 126; green, 211; blue, 33 }  ,draw opacity=1 ] [dash pattern={on 4.5pt off 4.5pt}]  (337.5,338.15) -- (371.81,360.07) ;
\draw [shift={(373.5,361.15)}, rotate = 212.57] [color={rgb, 255:red, 126; green, 211; blue, 33 }  ,draw opacity=1 ][line width=0.75]    (10.93,-3.29) .. controls (6.95,-1.4) and (3.31,-0.3) .. (0,0) .. controls (3.31,0.3) and (6.95,1.4) .. (10.93,3.29)   ;
\draw [color={rgb, 255:red, 126; green, 211; blue, 33 }  ,draw opacity=1 ] [dash pattern={on 4.5pt off 4.5pt}]  (328.5,324.65) -- (311.49,257.59) ;
\draw [shift={(311,255.65)}, rotate = 75.77] [color={rgb, 255:red, 126; green, 211; blue, 33 }  ,draw opacity=1 ][line width=0.75]    (10.93,-3.29) .. controls (6.95,-1.4) and (3.31,-0.3) .. (0,0) .. controls (3.31,0.3) and (6.95,1.4) .. (10.93,3.29)   ;
\draw [color={rgb, 255:red, 126; green, 211; blue, 33 }  ,draw opacity=1 ] [dash pattern={on 4.5pt off 4.5pt}]  (328.5,324.65) -- (266.24,289.14) ;
\draw [shift={(264.5,288.15)}, rotate = 29.7] [color={rgb, 255:red, 126; green, 211; blue, 33 }  ,draw opacity=1 ][line width=0.75]    (10.93,-3.29) .. controls (6.95,-1.4) and (3.31,-0.3) .. (0,0) .. controls (3.31,0.3) and (6.95,1.4) .. (10.93,3.29)   ;
\draw [color={rgb, 255:red, 126; green, 211; blue, 33 }  ,draw opacity=1 ] [dash pattern={on 4.5pt off 4.5pt}]  (328.5,324.65) -- (389.5,321.75) ;
\draw [shift={(391.5,321.65)}, rotate = 177.27] [color={rgb, 255:red, 126; green, 211; blue, 33 }  ,draw opacity=1 ][line width=0.75]    (10.93,-3.29) .. controls (6.95,-1.4) and (3.31,-0.3) .. (0,0) .. controls (3.31,0.3) and (6.95,1.4) .. (10.93,3.29)   ;
\draw [color={rgb, 255:red, 126; green, 211; blue, 33 }  ,draw opacity=1 ] [dash pattern={on 4.5pt off 4.5pt}]  (110.5,337.9) -- (144.81,359.82) ;
\draw [shift={(146.5,360.9)}, rotate = 212.57] [color={rgb, 255:red, 126; green, 211; blue, 33 }  ,draw opacity=1 ][line width=0.75]    (10.93,-3.29) .. controls (6.95,-1.4) and (3.31,-0.3) .. (0,0) .. controls (3.31,0.3) and (6.95,1.4) .. (10.93,3.29)   ;
\draw [color={rgb, 255:red, 126; green, 211; blue, 33 }  ,draw opacity=1 ] [dash pattern={on 4.5pt off 4.5pt}]  (101.5,324.4) -- (84.49,257.34) ;
\draw [shift={(84,255.4)}, rotate = 75.77] [color={rgb, 255:red, 126; green, 211; blue, 33 }  ,draw opacity=1 ][line width=0.75]    (10.93,-3.29) .. controls (6.95,-1.4) and (3.31,-0.3) .. (0,0) .. controls (3.31,0.3) and (6.95,1.4) .. (10.93,3.29)   ;
\draw [color={rgb, 255:red, 126; green, 211; blue, 33 }  ,draw opacity=1 ] [dash pattern={on 4.5pt off 4.5pt}]  (101.5,324.4) -- (162.5,321.5) ;
\draw [shift={(164.5,321.4)}, rotate = 177.27] [color={rgb, 255:red, 126; green, 211; blue, 33 }  ,draw opacity=1 ][line width=0.75]    (10.93,-3.29) .. controls (6.95,-1.4) and (3.31,-0.3) .. (0,0) .. controls (3.31,0.3) and (6.95,1.4) .. (10.93,3.29)   ;
\draw  [fill={rgb, 255:red, 0; green, 0; blue, 0 }  ,fill opacity=1 ] (288.5,320) .. controls (288.5,318.62) and (287.38,317.5) .. (286,317.5) .. controls (284.62,317.5) and (283.5,318.62) .. (283.5,320) .. controls (283.5,321.38) and (284.62,322.5) .. (286,322.5) .. controls (287.38,322.5) and (288.5,321.38) .. (288.5,320) -- cycle ;
\draw  [fill={rgb, 255:red, 0; green, 0; blue, 0 }  ,fill opacity=1 ] (76,265.5) .. controls (76,264.12) and (74.88,263) .. (73.5,263) .. controls (72.12,263) and (71,264.12) .. (71,265.5) .. controls (71,266.88) and (72.12,268) .. (73.5,268) .. controls (74.88,268) and (76,266.88) .. (76,265.5) -- cycle ;
\draw  [fill={rgb, 255:red, 0; green, 0; blue, 0 }  ,fill opacity=1 ] (302.5,265.5) .. controls (302.5,264.12) and (301.38,263) .. (300,263) .. controls (298.62,263) and (297.5,264.12) .. (297.5,265.5) .. controls (297.5,266.88) and (298.62,268) .. (300,268) .. controls (301.38,268) and (302.5,266.88) .. (302.5,265.5) -- cycle ;

\draw (100,178.4) node [anchor=north] [inner sep=0.75pt]  [font=\scriptsize]  {$a_{j}$};
\draw (89.17,195) node [anchor=north west][inner sep=0.75pt]  [font=\footnotesize] [align=left] {(a)};
\draw (133.5,169.4) node [anchor=north] [inner sep=0.75pt]  [font=\scriptsize]  {$a_{i}$};
\draw (167.67,146.9) node [anchor=north] [inner sep=0.75pt]  [font=\scriptsize]  {$a_{i-2}$};
\draw (61.83,280.13) node [anchor=north] [inner sep=0.75pt]  [font=\scriptsize]  {$original\ 2_{c}$};
\draw (326.5,178.4) node [anchor=north] [inner sep=0.75pt]  [font=\scriptsize]  {$a_{j}$};
\draw (361.5,167.4) node [anchor=north] [inner sep=0.75pt]  [font=\scriptsize]  {$a_{i}$};
\draw (392.5,147.07) node [anchor=north] [inner sep=0.75pt]  [font=\scriptsize]  {$a_{i-2}$};
\draw (451.17,147.4) node [anchor=north] [inner sep=0.75pt]  [font=\scriptsize]  {$newly\ introduced\ 2_{c}$};
\draw (250.5,141.9) node [anchor=north] [inner sep=0.75pt]  [font=\scriptsize]  {$a_{( i-2)^{*}}$};
\draw (240,90.4) node [anchor=north] [inner sep=0.75pt]  [font=\scriptsize]  {$a_{n}$};
\draw (14,89.4) node [anchor=north] [inner sep=0.75pt]  [font=\scriptsize]  {$a_{n}$};
\draw (180.5,368.8) node [anchor=north] [inner sep=0.75pt]  [font=\scriptsize]  {$a_{i-2} =a_{k_{r}}$};
\draw (126.5,135.23) node [anchor=north] [inner sep=0.75pt]  [font=\scriptsize]  {$original\ 2_{c}$};
\draw (100,398.3) node [anchor=north] [inner sep=0.75pt]  [font=\scriptsize]  {$a_{j}$};
\draw (131.5,389.8) node [anchor=north] [inner sep=0.75pt]  [font=\scriptsize]  {$a_{i}$};
\draw (14,309.3) node [anchor=north] [inner sep=0.75pt]  [font=\scriptsize]  {$a_{n}$};
\draw (181,340.4) node [anchor=north] [inner sep=0.75pt]  [font=\scriptsize]  {$a_{k_{r}^{*}}$};
\draw (189,315.3) node [anchor=north] [inner sep=0.75pt]  [font=\scriptsize]  {$a_{k_{r-1}}$};
\draw (181.5,280.9) node [anchor=north] [inner sep=0.75pt]  [font=\scriptsize]  {$a_{k_{r-1}^{*}}$};
\draw (81,230.8) node [anchor=north] [inner sep=0.75pt]  [font=\scriptsize]  {$a_{k_{1}}$};
\draw (45.5,244.4) node [anchor=north] [inner sep=0.75pt]  [font=\scriptsize]  {$a_{k_{1}^{*}}$};
\draw (22.5,271.8) node [anchor=north] [inner sep=0.75pt]  [font=\scriptsize]  {$a_{k_{0}}$};
\draw (117,227.4) node [anchor=north] [inner sep=0.75pt]  [font=\scriptsize]  {$a_{k_{2}^{*}}$};
\draw (406.5,368.8) node [anchor=north] [inner sep=0.75pt]  [font=\scriptsize]  {$a_{i-2} =a_{k_{r}}$};
\draw (326,398.3) node [anchor=north] [inner sep=0.75pt]  [font=\scriptsize]  {$a_{j}$};
\draw (357.5,389.8) node [anchor=north] [inner sep=0.75pt]  [font=\scriptsize]  {$a_{i}$};
\draw (240,309.3) node [anchor=north] [inner sep=0.75pt]  [font=\scriptsize]  {$a_{n}$};
\draw (407,340.4) node [anchor=north] [inner sep=0.75pt]  [font=\scriptsize]  {$a_{k_{r}^{*}}$};
\draw (415,315.3) node [anchor=north] [inner sep=0.75pt]  [font=\scriptsize]  {$a_{k_{r-1}}$};
\draw (407.5,280.9) node [anchor=north] [inner sep=0.75pt]  [font=\scriptsize]  {$a_{k_{r-1}^{*}}$};
\draw (307,230.8) node [anchor=north] [inner sep=0.75pt]  [font=\scriptsize]  {$a_{k_{1}}$};
\draw (271.5,244.4) node [anchor=north] [inner sep=0.75pt]  [font=\scriptsize]  {$a_{k_{1}^{*}}$};
\draw (248.5,271.8) node [anchor=north] [inner sep=0.75pt]  [font=\scriptsize]  {$a_{k_{0}}$};
\draw (343,227.4) node [anchor=north] [inner sep=0.75pt]  [font=\scriptsize]  {$a_{k_{2}^{*}}$};
\draw (243.7,337.8) node [anchor=north] [inner sep=0.75pt]  [font=\scriptsize]  {$a_{k_{0}^{*}}$};
\draw (337.17,324.8) node [anchor=north] [inner sep=0.75pt]  [font=\scriptsize]  {$newly\ introduced\ 2_{c}$};
\draw (110.17,324.55) node [anchor=north] [inner sep=0.75pt]  [font=\scriptsize]  {$newly\ introduced\ 2_{c}$};
\draw (317.17,195) node [anchor=north west][inner sep=0.75pt]  [font=\footnotesize] [align=left] {(b)};
\draw (89.17,415) node [anchor=north west][inner sep=0.75pt]  [font=\footnotesize] [align=left] {(c)};
\draw (317.17,415) node [anchor=north west][inner sep=0.75pt]  [font=\footnotesize] [align=left] {(d)};

\end{tikzpicture}